\documentclass{amsart}
\usepackage{tikz}
\usepackage{ytableau}

\usetikzlibrary{arrows, positioning, shapes}
\usepackage{youngtab}
\usetikzlibrary{cd}
\usetikzlibrary{arrows}
\usepackage{pgf}
\usepackage{mathrsfs}
\usepackage{rank-2-roots}
\usepackage{advdate}
\usepackage{fullpage}
\usepackage{gensymb}
\usepackage[colorlinks=true, linkcolor=red, citecolor=teal, urlcolor=teal]{hyperref}
\usepackage{hyperref}
\usepackage{cleveref}
\usetikzlibrary{arrows.meta}
\usepackage{tabularx}
\usepackage{float}

\usepackage{amsmath, amssymb, amsfonts, amsthm}

\usepackage{graphicx} 
\usepackage{booktabs} 

\newtheorem{theorem}{Theorem}
\newtheorem{definition}[theorem]{Definition}
\newtheorem{lemma}[theorem]{Lemma}

\newtheorem{corollary}[theorem]{Corollary}

\title{Evacuation of rectangular standard Young tableaux corresponds to reflection of $\mathfrak{sl}_n$ webs}
\author{Lucas Adams Cowan \and Ronja Eilfort \and  Kerry Seekamp \and Julianna Tymoczko}

\begin{document}
\begin{abstract}
 Web graphs form a family of planar directed graphs with boundary that can be used to model quantum $\mathfrak{sl}_n$-invariant vectors.  Standard Young tableaux on an $n \times k$ rectangle naturally index a basis for $\mathfrak{sl}_n$ web graphs.  We prove that evacuation of the tableau $T$ corresponds to reflection of the associated web graph $w_T$ up to equivalence under a specific set of edge-flip relations.  This extends a result of Patrias and Pechenik for the cases $n=2,3$ and mirrors analogous results about rotation of web graphs corresponding to promotion of tableau by Peterson-Pylyavskyy-Rhoades for $n=3$ and Gaetz-Pechenik-Pfannerer-Striker-Swanson for $n=4$.  We use an intermediate object called a \textit{multicolored noncrossing matching}, which is closely related to the notion of strandings recently introduced by Russell and the fourth author.
\end{abstract}

\maketitle

\section{Introduction}

Tableaux are well-studied objects that visually and intuitively index algebraic structures, facilitating elegant correspondences within algebra and combinatorics in fields like representation theory, Schubert calculus, and symmetric function theory. Operations in these other fields are encoded by combinatorial procedures on tableaux, e.g. Sch\"{u}tzenberger's promotion  \cite{SCHUTZENBERGER1963, SCHUTZENBERGER197273}.  We study evacuation of tableaux, which is a kind of repeated promotion that is arguably even more important than promotion itself \cite{EdeGre87, Hai92, Ste96, Sta09}. 

Our main result demonstrates that evacuation of rectangular standard Young tableaux is equivalent to reflection of a type of graph called \emph{web graphs}. Broadly speaking, $\mathfrak{sl}_n$ web graphs are  directed, edge-labeled, trivalent plane graphs with boundary that satisfy a certain conservation of flow mod $n$ at each interior vertex.   Web graphs describe the morphisms in a diagrammatic category for  $U_q(\mathfrak{sl}_n)$-representations; equivalently, web graphs give a concrete recipe to construct $U_q(\mathfrak{sl}_n)$-invariant vectors. Webs arise in a variety of mathematical contexts, including combinatorics, algebraic geometry, representation theory, knot theory, cluster algebras, mathematical physics, and categorification. They are known to be closely connected to the combinatorics of tableaux: for instance, rotation of $\mathfrak{sl}_n$ webs intertwines with promotion of rectangular standard Young tableaux for $n=2, 3$ \cite{PPR09, Tymoczko2012}, extended recently by Gaetz et alia to $\mathfrak{sl}_4$ webs \cite{GPPPS}.

The work in this paper on evacuation and reflection streamlines and generalizes an earlier result of Patrias-Pechenik, which proved a similar result in the case of $\mathfrak{sl}_n$ webs when $n=2$ and $n=3$ in terms of an involution described using careful planar geometry~\cite{PatriasPechenik2023}.   Our work relies deeply on Fontaine's model for $\mathfrak{sl}_n$ web graphs \cite{Fontaine2012}, for which a complete set of relations was only recently proven \cite{RT2025}. The strategy of our proof is to build webs from standard Young tableaux through an intermediate object called a \emph{multicolored noncrossing matching} that generalizes noncrossing matchings (which are themselves ubiquitous within algebraic combinatorics).  
Multicolored noncrossing matchings appear in the context of web graphs as \emph{strandings}: certain directed colored paths inside web graphs that encode  global structure, provide a new framework to analyze the invariant vector associated to a web graph, and invite new combinatorial operations on webs (see \cite{RT2025}).   The arguments of this paper do not require the precise definition of strandings so we omit it; however, the multicolored noncrossing matchings here are, in fact, strandings.

Our paper is structured as follows. Section~\ref{SEC:Tableaux and evacuation} defines tableaux and evacuation, making sense of the first column of Figure~\ref{fig: intro example}. Section~\ref{SEC:MC NCM and Reflection} introduces multicolored noncrossing matchings, illustrates how to build a multicolored noncrossing matching $\mathcal{M}^T$ from a tableau $T$, and describes reflection $\varphi(\mathcal{M})$ of a multicolored noncrossing matching (which both reflects the matching and exchanges colors in the matching), leading to the second column of Figure~\ref{fig: intro example}. Section~\ref{SEC:webs and reflection} defines web graphs, provides an algorithm for producing a web graph $w_T$ from a tableau $T$ via its multicolored noncrossing matching, and verifies that reflection over a  vertical line  is an involution of webs $w \leftrightarrow \varphi(w)$, producing the third column of Figure~\ref{fig: intro example}.  
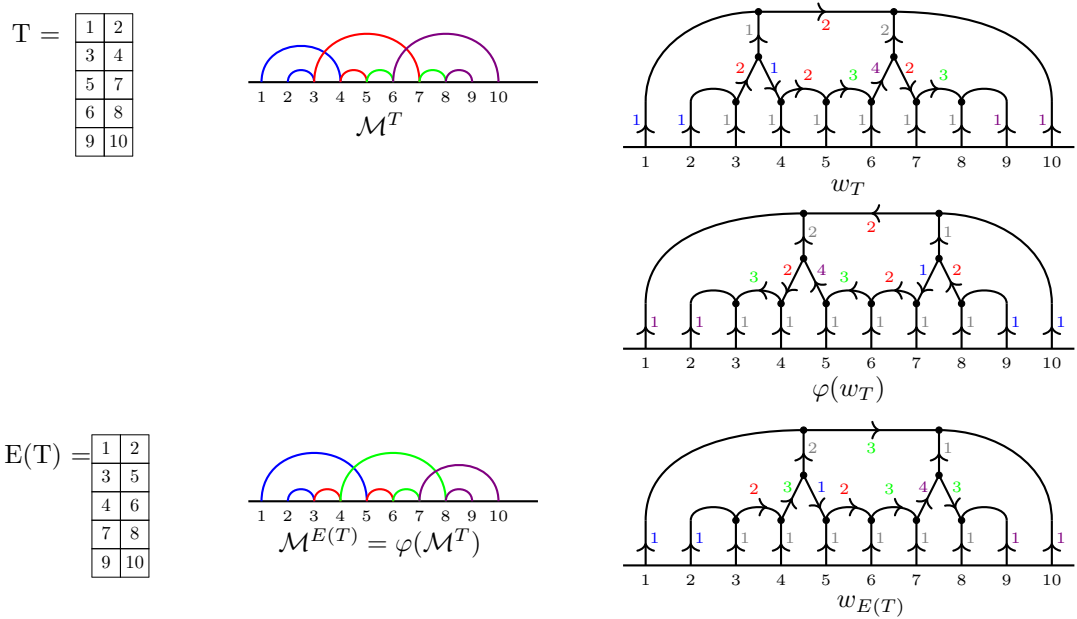
\begin{figure}[h!]
\begin{center}
\begin{tabular}{ccc}
$
\raisebox{.2in}{T =  \scalebox{0.7}{
\begin{ytableau}
                1 & 2 \\
                3 & 4  \\
                5 & 7 \\
                6 & 8 \\
                9 & 10 \\
            \end{ytableau}
     } }
     $
     & \hspace{0.25in}
     \begin{tikzpicture}[baseline=(current bounding box.center)]
    \begin{scope}[xscale=0.35, yscale=0.55]
        \node at (5.5,-1) {$\mathcal{M}^T$};
        \draw[thick] (.5,0) -- (11.5,0);
        \foreach \i in {1,...,10} \node[below, font=\tiny] (\i) at (\i,0) {\i};
        
        \draw[blue, thick] (1) to[out=90,in=90] (4);
        \draw[blue, thick] (2) to[out=90,in=90] (3);
        \draw[red, thick] (4) to[out=90,in=90] (5);
        \draw[red, thick] (3) to[out=90,in=90] (7);
        \draw[green, thick] (5) to[out=90,in=90] (6);
        \draw[green, thick] (7) to[out=90,in=90] (8);
        \draw[violet, thick] (8) to[out=90,in=90] (9);
        \draw[violet, thick] (6) to[out=90,in=90] (10);
        
    \end{scope}
\end{tikzpicture} \hspace{0.25in}
&
\raisebox{-0.65in}{\begin{tikzpicture}[xscale=0.6, yscale=0.6]
     \draw[thick] (.5,0) -- (10.5,0);
         \foreach \i in {1,...,10}
            \node[below, font=\tiny] (\i) at (\i,0) {\i};

    \node at (5.5,-.9) {$w_T$};

    \filldraw[black] (3,1) circle (2pt);
    \filldraw[black] (4,1) circle (2pt);
    \filldraw[black] (5,1) circle (2pt);
    \filldraw[black] (6,1) circle (2pt);
    \filldraw[black] (7,1) circle (2pt);
    \filldraw[black] (8,1) circle (2pt);
    \filldraw[black] (3.5,2) circle (2pt);
    \filldraw[black] (3.5,3) circle (2pt);
    \filldraw[black] (6.5,2) circle (2pt);
    \filldraw[black] (6.5,3) circle (2pt);

    \draw[thick] (2,1) to[out=90,in=90] (3,1);
    \draw[thick] (8,1) to[out=90,in=90] (9,1);

    \draw[thick] (1,1) to[out=90,in=180] (3.5,3);
    \draw[thick] (10,1) to[out=90,in=0] (6.5,3);

    \begin{scope}[thick,decoration={
    markings,
    mark=at position 0.5 with {\arrow{>}}}
    ] 

    \foreach \i in {1,...,10}
            \draw[thick, postaction={decorate}] (\i,0) -- (\i,1);
    \draw[thick, postaction={decorate}] (3, 1) -- (3.5, 2);
    \draw[thick, postaction={decorate}] (3.5, 2) -- (3.5, 3);
    \draw[thick, postaction={decorate}] (6, 1) -- (6.5, 2);
    \draw[thick, postaction={decorate}] (6.5, 2) -- (6.5, 3);

    \draw[thick, postaction={decorate}] (3.5, 3) -- (6.5, 3);

    \draw[thick, postaction={decorate}] (4,1) to[out=90,in=90] (5,1);
    \draw[thick, postaction={decorate}] (5,1) to[out=90,in=90] (6,1);
    \draw[thick, postaction={decorate}] (7,1) to[out=90,in=90] (8,1);
\end{scope}

    \node[blue, font=\tiny] at (.8,.6) {$1$};
    \node[blue, font=\tiny] at (1.8,.6) {$1$};
    \node[gray, font=\tiny] at (2.8,.6) {$1$};
    \node[gray, font=\tiny] at (3.8,.6) {$1$};
    \node[gray, font=\tiny] at (4.8,.6) {$1$};
    \node[gray, font=\tiny] at (5.8,.6) {$1$};
    \node[gray, font=\tiny] at (6.8,.6) {$1$};
    \node[gray, font=\tiny] at (7.8,.6) {$1$};
    \node[violet, font=\tiny] at (8.8,.6) {$1$};
    \node[violet, font=\tiny] at (9.8,.6) {$1$};

    \node[red, font=\tiny] at (3.1,1.75) {$2$};
    \node[blue, font=\tiny] at (3.85,1.75) {$1$};
    \node[red, font=\tiny] at (4.6,1.6) {$2$};
    \node[green, font=\tiny] at (5.6,1.6) {$3$};
    \node[violet, font=\tiny] at (6.1,1.75) {$4$};
    \node[red, font=\tiny] at (6.85,1.75) {$2$};
    \node[green, font=\tiny] at (7.6,1.6) {$3$};

    \node[gray, font=\tiny] at (3.3,2.6) {$1$};
    \node[gray, font=\tiny] at (6.3,2.6) {$2$};

    \node[red, font=\tiny] at (5,2.7) {$2$};

    \begin{scope}[thick,decoration={
    markings,
    mark=at position 0.4 with {\arrow{<}}}
    ] 
    \draw[thick, postaction={decorate}] (4, 1) -- (3.5, 2);
    \draw[thick, postaction={decorate}] (7, 1) -- (6.5, 2);

    \end{scope}
\end{tikzpicture}
}
     
\\      
& &
\raisebox{0in}{\begin{tikzpicture}[xscale=-0.6, yscale=0.6]
     \draw[thick] (.5,0) -- (10.5,0);
         \foreach \i in {1,...,10}
            \node[below, font=\tiny] (\i) at (11-\i,0) {\i};

    \node at (5.5,-.9) {$\varphi(w_T)$};

    \filldraw[black] (3,1) circle (2pt);
    \filldraw[black] (4,1) circle (2pt);
    \filldraw[black] (5,1) circle (2pt);
    \filldraw[black] (6,1) circle (2pt);
    \filldraw[black] (7,1) circle (2pt);
    \filldraw[black] (8,1) circle (2pt);
    \filldraw[black] (3.5,2) circle (2pt);
    \filldraw[black] (3.5,3) circle (2pt);
    \filldraw[black] (6.5,2) circle (2pt);
    \filldraw[black] (6.5,3) circle (2pt);

    \draw[thick] (2,1) to[out=90,in=90] (3,1);
    \draw[thick] (8,1) to[out=90,in=90] (9,1);

    \draw[thick] (1,1) to[out=90,in=180] (3.5,3);
    \draw[thick] (10,1) to[out=90,in=0] (6.5,3);

    \begin{scope}[thick,decoration={
    markings,
    mark=at position 0.5 with {\arrow{>}}}
    ] 

    \foreach \i in {1,...,10}
            \draw[thick, postaction={decorate}] (\i,0) -- (\i,1);
    \draw[thick, postaction={decorate}] (3, 1) -- (3.5, 2);
    \draw[thick, postaction={decorate}] (3.5, 2) -- (3.5, 3);
    \draw[thick, postaction={decorate}] (6, 1) -- (6.5, 2);
    \draw[thick, postaction={decorate}] (6.5, 2) -- (6.5, 3);

    \draw[thick, postaction={decorate}] (3.5, 3) -- (6.5, 3);

    \draw[thick, postaction={decorate}] (4,1) to[out=90,in=90] (5,1);
    \draw[thick, postaction={decorate}] (5,1) to[out=90,in=90] (6,1);
    \draw[thick, postaction={decorate}] (7,1) to[out=90,in=90] (8,1);
\end{scope}

    \node[blue, font=\tiny] at (.8,.6) {$1$};
    \node[blue, font=\tiny] at (1.8,.6) {$1$};
    \node[gray, font=\tiny] at (2.8,.6) {$1$};
    \node[gray, font=\tiny] at (3.8,.6) {$1$};
    \node[gray, font=\tiny] at (4.8,.6) {$1$};
    \node[gray, font=\tiny] at (5.8,.6) {$1$};
    \node[gray, font=\tiny] at (6.8,.6) {$1$};
    \node[gray, font=\tiny] at (7.8,.6) {$1$};
    \node[violet, font=\tiny] at (8.8,.6) {$1$};
    \node[violet, font=\tiny] at (9.8,.6) {$1$};

    \node[red, font=\tiny] at (3.1,1.75) {$2$};
    \node[blue, font=\tiny] at (3.85,1.75) {$1$};
    \node[red, font=\tiny] at (4.6,1.6) {$2$};
    \node[green, font=\tiny] at (5.6,1.6) {$3$};
    \node[violet, font=\tiny] at (6.1,1.75) {$4$};
    \node[red, font=\tiny] at (6.85,1.75) {$2$};
    \node[green, font=\tiny] at (7.6,1.6) {$3$};

    \node[gray, font=\tiny] at (3.3,2.6) {$1$};
    \node[gray, font=\tiny] at (6.3,2.6) {$2$};

    \node[red, font=\tiny] at (5,2.7) {$2$};

    \begin{scope}[thick,decoration={
    markings,
    mark=at position 0.4 with {\arrow{<}}}
    ] 
    \draw[thick, postaction={decorate}] (4, 1) -- (3.5, 2);
    \draw[thick, postaction={decorate}] (7, 1) -- (6.5, 2);

    \end{scope}
\end{tikzpicture}
} \\
$       
\raisebox{.2in}{ E(T) =\scalebox{0.7}{\begin{ytableau}
                1 & 2 \\
                3 & 5  \\
                4 & 6 \\
                7 & 8 \\
                9 & 10 \\
            \end{ytableau}
            }
            }
$

& 
\begin{tikzpicture}[baseline=(current bounding box.center)]
    \begin{scope}[xscale=0.35, yscale=0.55]
        \node at (5.5,-1) {$\mathcal{M}^{E(T)} = \varphi(\mathcal{M}^T)$};
        \draw[thick] (.5,0) -- (11.5,0);
        \foreach \i in {1,...,10} \node[below, font=\tiny] (\i) at (\i,0) {\i};
        
        \draw[blue, thick] (1) to[out=90,in=90] (5);
        \draw[blue, thick] (2) to[out=90,in=90] (3);
        \draw[red, thick] (3) to[out=90,in=90] (4);
        \draw[red, thick] (5) to[out=90,in=90] (6);
        \draw[green, thick] (6) to[out=90,in=90] (7);
        \draw[green, thick] (4) to[out=90,in=90] (8);
        \draw[violet, thick] (8) to[out=90,in=90] (9);
        \draw[violet, thick] (7) to[out=90,in=90] (10);
        
    \end{scope}
\end{tikzpicture}
& 
\raisebox{-0.65in}{\begin{tikzpicture}[xscale=-0.6, yscale=0.6]
     \draw[thick] (.5,0) -- (10.5,0);
         \foreach \i in {1,...,10}
            \node[below, font=\tiny] (\i) at (11-\i,0) {\i};

    \node at (5,-.9) {$w_{E(T)}$};

    \filldraw[black] (3,1) circle (2pt);
    \filldraw[black] (4,1) circle (2pt);
    \filldraw[black] (5,1) circle (2pt);
    \filldraw[black] (6,1) circle (2pt);
    \filldraw[black] (7,1) circle (2pt);
    \filldraw[black] (8,1) circle (2pt);
    \filldraw[black] (3.5,2) circle (2pt);
    \filldraw[black] (3.5,3) circle (2pt);
    \filldraw[black] (6.5,2) circle (2pt);
    \filldraw[black] (6.5,3) circle (2pt);

    \draw[thick] (2,1) to[out=90,in=90] (3,1);
    \draw[thick] (8,1) to[out=90,in=90] (9,1);

    \draw[thick] (1,1) to[out=90,in=180] (3.5,3);
    \draw[thick] (10,1) to[out=90,in=0] (6.5,3);

    \begin{scope}[thick,decoration={
    markings,
    mark=at position 0.5 with {\arrow{<}}}
    ] 
    \draw[thick, postaction={decorate}] (3.5, 3) -- (6.5, 3);

    \draw[thick, postaction={decorate}] (4,1) to[out=90,in=90] (5,1);
    \draw[thick, postaction={decorate}] (5,1) to[out=90,in=90] (6,1);
    \draw[thick, postaction={decorate}] (7,1) to[out=90,in=90] (8,1);
\end{scope}

    \node[blue, font=\tiny] at (9.8,.6) {$1$};
    \node[blue, font=\tiny] at (8.8,.6) {$1$};
    \node[gray, font=\tiny] at (2.8,.6) {$1$};
    \node[gray, font=\tiny] at (3.8,.6) {$1$};
    \node[gray, font=\tiny] at (4.8,.6) {$1$};
    \node[gray, font=\tiny] at (5.8,.6) {$1$};
    \node[gray, font=\tiny] at (6.8,.6) {$1$};
    \node[gray, font=\tiny] at (7.8,.6) {$1$};
    \node[violet, font=\tiny] at (1.8,.6) {$1$};
    \node[violet, font=\tiny] at (.8,.6) {$1$};

    \node[green, font=\tiny] at (3.1,1.75) {$3$};
    \node[violet, font=\tiny] at (3.85,1.75) {$4$};
    \node[green, font=\tiny] at (4.6,1.7) {$3$};
    \node[red, font=\tiny] at (5.6,1.7) {$2$};
    \node[blue, font=\tiny] at (6.1,1.75) {$1$};
    \node[green, font=\tiny] at (6.85,1.75) {$3$};
    \node[red, font=\tiny] at (7.6,1.7) {$2$};

    \node[gray, font=\tiny] at (3.3,2.6) {$1$};
    \node[gray, font=\tiny] at (6.3,2.6) {$2$};

    \node[green, font=\tiny] at (5,2.6) {$3$};

    \begin{scope}[thick,decoration={
    markings,
    mark=at position 0.5 with {\arrow{>}}}
    ] 
    \foreach \i in {1,...,10}
            \draw[thick, postaction={decorate}] (\i,0) -- (\i,1);
    \draw[thick, postaction={decorate}] (4, 1) -- (3.5, 2);
    \draw[thick, postaction={decorate}] (7, 1) -- (6.5, 2);

    \draw[ thick, postaction={decorate}] (3.5, 2) -- (3.5, 3);
    \draw[ thick, postaction={decorate}] (6.5, 2) -- (6.5, 3);
    \end{scope}

    \begin{scope}[thick,decoration={
    markings,
    mark=at position 0.75 with {\arrow{>}}}
    ] 
    \draw[thick, postaction={decorate}] (3.5, 2) -- (3, 1);
    \draw[thick, postaction={decorate}] (6.5, 2) -- (6, 1);

    \end{scope}

\end{tikzpicture}
}
\end{tabular}
\end{center}
\caption{Example of a tableau $T$ and its evacuation, plus the multicolored noncrossing matchings and  web graphs associated to each.  The second row has the reflection $\varphi(w_T)$ of $w_T$ over the line $x=5.5$ which differs from $w_{E(T)}$ only in a few edge directions and weights.} \label{fig: intro example}
\end{figure}

Finally, Section~\ref{SEC: Evacuation and reflection} shows that each square in Figure~\ref{commutative} commutes. To do this, we rely on a result of Patrias-Pechenik comparing the rotation of a rectangular tableau with its evacuation \cite{PatriasPechenik2023}. The main subtlety of our result is that $\mathfrak{sl}_n$ web graphs have edge-weights in the set $\{1, 2, \ldots, n-1\}$ and these weights are in general not invariant under reflection.  In other words, the reflection $\varphi(w_T)$ does not have exactly the same edge weights and directions as $w_{E(T)}$ in Figure~\ref{fig: intro example}.  

Thus, our main result---Lemma~\ref{lemma: right square commutes}---actually proves several claims.  The first is that the reflected web graph $\varphi(w_T)$ agrees with the web graph $w_{E(T)}$ associated to the evacuation $E(T)$ as \emph{undirected, unweighted graphs}.  Since edges in $\mathfrak{sl}_3$ and $\mathfrak{sl}_2$ webs are generally weighted and directed according to conventions that ensure each edge weight is $1$, this recovers Patrias-Pechenik's results. In general, properly speaking, web graphs are equivalence classes under certain web relations; the conventions for $n=2, 3$ select a particular web graph in each equivalence class.  The second claim in Lemma~\ref{lemma: right square commutes} is that $\varphi(w_T)$ and $w_{E(T)}$ are in the same equivalence class of $\mathfrak{sl}_n$ web graphs, and in fact differ only by the  equivalence relation \emph{edge-flipping}.  Finally, the third claim in Lemma~\ref{lemma: right square commutes} identifies precisely which edges must be flipped to transform $\varphi(w_T)$ into $w_{E(T)}$ as web graphs (not equivalence classes). As Figure~\ref{fig: intro example} suggests, the flipped edges are, loosely speaking, the interior edges that are not vertical. Corollary~\ref{corollary: sl3 and sl4 web conventions} confirms that using the usual conventions for edge weights and directions of $\mathfrak{sl}_3$ or $\mathfrak{sl}_4$ webs means that $\varphi(w_T)=w_{E(T)}$ as web graphs (not as equivalence classes).

\begin{figure}[h]
    \begin{center}
        \begin{tikzcd}
        \textup{Young tableaux} \arrow[r]\arrow[d,"E"] & \textup{multicolored NCMs} \arrow[r] \arrow[d,"\varphi"] & \textup{webs under edge-flip equivalence} \arrow[d,"\varphi"] \\
        \textup{Young tableaux} \arrow[r] & \textup{multicolored NCMs} \arrow[r] & \textup{webs under edge-flip equivalence}
        \end{tikzcd}
    \end{center} 
    \caption{Commutative diagram describing our argument, where $E$ denotes the evacuation map,  $\varphi$ denotes reflection maps, and NCM abbreviates noncrossing matching} \label{commutative}
\end{figure}

\subsection{List of notation}

For the reader's convenience, we list all our notation below.
\begin{itemize}
    \item $T$ is a standard Young tableaux with $N$ boxes, $n$ rows, and $k$ columns, so $N=nk$.
    \item $E(T)$ is the evacuation of the tableau $T$ and $\rho(T)$ is the $180$ degree rotation of the tableau $T$.
    \item An ordered pair $(i,j)$ is an arc with initial point $i$ and terminal point $j$.   Each $x\in \{1,2,\ldots,n-1\}$ denotes a different color.  The (colored) matching $\mathcal{M}_x$  consists of a collection of arcs $(i,j)$ colored $x$.
    \item $\mathcal{M} = (\mathcal{M}_1, \mathcal{M}_2, \ldots, \mathcal{M}_{n-1})$ is a multicolored matching, with matchings $\mathcal{M}_x$ of colors $x=1,  \ldots, n-1$. 
    \item $\mathcal{M}^T$ is the multicolored noncrossing matching (NCM) obtained from the standard tableau $T$. The reflection of a multicolored matching $\mathcal{M}$ is denoted $\varphi(\mathcal{M})$. The multicolored NCM obtained from the rotated tableau $\rho(T)$ is denoted $\mathcal{M}^{\rho(T)}$. 
    \item $w$ is a web graph, with vertices denoted $v$ and edges denoted $e$. The edge weight of $e$ is $\ell(e)$.
    \item $w_\mathcal{M}$ is the web graph obtained from the multicolored NCM $\mathcal{M}$ and
 $w_T$ is the web graph obtained from $w_{\mathcal{M}_T}$ according to Corollary~\ref{corollary: webs from tableau in standard form}.
 $\varphi(w)$ is the reflection of the web graph $w$.

\end{itemize}

\subsection{Acknowledgments} The first two authors were partially supported by the Smith College SURF program.  The fourth author was partially supported by NSF grant DMS-2349088.  The authors gratefully acknowledge Heather M.~Russell and Sharon Spaulding. 

\section{Tableaux and evacuation} \label{SEC:Tableaux and evacuation}

 In this section, we define the main tableau operation described in this paper, namely, evacuation. We also summarize the results of Patrias-Pechenik concerning rotation of webs and evacuation of Young tableaux, which lays the foundation for our main argument.

\begin{definition}
\label{definition: Young tableaux}
A \emph{Young diagram} is an array of left-justified boxes corresponding to an integer partition $\lambda=(\lambda_1\geq\lambda_2\geq\ldots\geq\lambda_n\geq0)$, where $\lambda_i$ denotes the number of boxes in the $i^{th}$ row from the top.

A \emph{rectangular standard Young tableau} (SYT) is a bijective filling of a rectangular Young diagram with integers $1,2,\ldots N$ so that entries strictly increase left to right in  rows and top to bottom in  columns.
\end{definition}

This manuscript focuses on Young diagrams of rectangular shape, meaning we consider Young diagrams with $N$ boxes of shape $n \times k$, where $n$ is the number of rows and $k$ is the length of each row, or equivalently the number of columns.  So in our case, the total number of boxes in a rectangular Young diagram satisfies $N=nk$. The tableaux labeled $T$ and $E(T)$ in Figures~\ref{fig:sl3evac} and~\ref{fig:evac-flip} are examples of  standard Young tableaux. 

Next, we define promotion, which is a key step in the operation of evacuation.

\begin{definition} \label{definition: promotion} 
A \emph{jeu de taquin slide} in a configuration like $\begin{array}{|c|c|} \cline{1-2} \bullet & b \\ \cline{1-2} a & \multicolumn{1}{c}{\hspace{0.5em}} \\ \cline{1-1} \multicolumn{2}{c}{}\end{array}$ is a slide of one of the entries $a$ or $b$ up or to the left, respectively, that moves the lesser value into the empty cell marked $\bullet$. 

Suppose $T$ is a rectangular standard Young tableau with $N$ boxes.  To obtain the \emph{promotion} of $T$, remove the entry of the cell in the most northwest corner of $T$. Perform a sequence of jeu de taquin slides, continuing until the empty cell is in the southeast corner. Then subtract one from all entries and fill the empty cell with the integer $N$. 
\end{definition}

Since $T$ is standard, the entry in its most northwest corner will always be the smallest entry of $T$.

The central operation studied in this article is an involution called evacuation, a type of repeated promotion on standard Young tableaux. We give an example of evacuation in Figure~\ref{fig:sl3evac}.

\begin{definition} \label{definition: evacuation}
Suppose $T$ is a rectangular standard Young tableau with $N$ boxes. To obtain its \emph{evacuation} $E(T)$, repeat the following steps until all entries are negative:
\begin{enumerate}
    \item remove the entry of the cell in the northwest corner of $T$,
    \item perform jeu de taquin slides until the empty cell is in the most southeast position possible, and
    \item fill the empty cell with the negative of the  entry removed in Step 1, and fix it in all future iterations.
\end{enumerate} 
When all cells are negative, add $N+1$ to all entries. 
\end{definition}

\begin{figure}
    \centering
    $T=$
    \ytableausetup{centertableaux}
\begin{ytableau}
1 & 3 \\
2 & 4 \\
5 & 6 \\
\end{ytableau}
\hspace{1in}
$E(T)=$
\begin{ytableau}
1 & 2 \\
3 & 5 \\
4 & 6 \\
\end{ytableau}
\vspace{.5cm}

    \ytableausetup{centertableaux}
\begin{ytableau}
1 & 3 \\
2 & 4 \\
5 & 6 \\
\end{ytableau}
\begin{tikzpicture}
    \draw [-stealth](0,0) -- (.5,0);
\end{tikzpicture}
\begin{ytableau}
\bullet & 3 \\
2 & 4 \\
5 & 6 \\
\end{ytableau}
\begin{tikzpicture}
    \draw [-stealth](0,0) -- (.5,0);
\end{tikzpicture}
\begin{ytableau}
2 & 3 \\
4 & \bullet \\
5 & 6 \\
\end{ytableau}
\begin{tikzpicture}
    \draw [-stealth](0,0) -- (.5,0);
\end{tikzpicture}
\begin{ytableau}
2 & 3 \\
4 & 6 \\
5 & \bullet \\
\end{ytableau}
\begin{tikzpicture}
    \draw [-stealth](0,0) -- (.5,0);
\end{tikzpicture}
\begin{ytableau}
2 & 3 \\
4 & 6 \\
5 & *(green)-1 \\
\end{ytableau}

\vspace{.5cm}

\ytableausetup{centertableaux}

\begin{ytableau}
3 & 6 \\
4 & *(green)-2\\
5 & *(green)-1 \\
\end{ytableau}
\begin{tikzpicture}
    \draw [-stealth](0,0) -- (.25,0);
    \draw [-stealth](.3,0) -- (.5,0);
\end{tikzpicture}
\begin{ytableau}
4 & 6 \\
5 & *(green)-2\\
*(green)-3 & *(green)-1 \\
\end{ytableau}
\begin{tikzpicture}
    \draw [-stealth](0,0) -- (.25,0);
    \draw [-stealth](.3,0) -- (.5,0);
\end{tikzpicture}
\begin{ytableau}
5 & 6 \\
*(green)-4 & *(green)-2\\
*(green)-3 & *(green)-1 \\
\end{ytableau}
\begin{tikzpicture}
    \draw [-stealth](0,0) -- (.25,0);
    \draw [-stealth](.3,0) -- (.5,0);
\end{tikzpicture}
\begin{ytableau}
6 & *(green)-5 \\
*(green)-4 & *(green)-2\\
*(green)-3 & *(green)-1 \\
\end{ytableau}
\begin{tikzpicture}
    \draw [-stealth](0,0) -- (.25,0);
    \draw [-stealth](.3,0) -- (.5,0);
\end{tikzpicture}
\begin{ytableau}
*(green)-6 & *(green)-5 \\
*(green)-4 & *(green)-2\\
*(green)-3 & *(green)-1 \\
\end{ytableau}

    \caption{Evacuation of a SYT of shape $3\times2$. The first row shows the original tableau $T$ (left) and its evacuation $E(T)$ (right). The second row shows the removal of the entry of the cell in the northwest corner \textemdash indicating the empty cell with $\bullet$ \textemdash followed by jeu de taquin slides. The third row shows the intermediate SYT after each iteration of jeu de taquin slides, indicating fixed, or negative, cells with green.}
    \label{fig:sl3evac}
\end{figure}
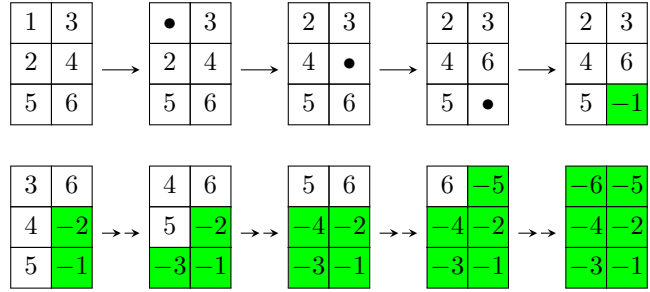

\subsection{Patrias--Pechenik: evacuation of rectangular tableaux is rotation and reversal}
 One key step of our proof uses a result of Patrias-Pechenik, which established that the evacuation of rectangular row-strict tableaux can be found by rotating $180$ degrees and replacing each entry $i$ with $N+1-i$. Since row-strict tableaux weakly increase left to right on each row, row-strict is a weaker condition than standard; thus, the result applies to the tableaux in this article, namely, rectangular standard Young tableaux.

\begin{lemma}[Patrias-Pechenik] \label{lemma: Patrias Pechenik}
Let $T$ be a row-strict tableau of rectangular shape $\lambda = (k,k,...,k)$ with maximum entry $N=nk$. Then $E(T)$ is given by rotating $T$ by $180\degree$ and replacing each entry according to the map $i\rightarrow N+ 1- i$. 
\label{PP}
\end{lemma}

We denote the tableaux obtained from $T$ by rotating $180$ degrees with $\rho(T)$. Figure~\ref{fig:evac-flip} gives an example.

\begin{figure}[H]
    \centering
    $T=$
    \ytableausetup{centertableaux}
\begin{ytableau}
1 & 3 & 5\\
2 & 4& 8\\
6 & 9 & 10\\
7 & 11 & 12\\
\end{ytableau}
$\rightsquigarrow$ rotate $180\degree \rightsquigarrow$
$\rho(T)=$ \begin{ytableau}
12 & 11 & 7 \\
10 & 9 & 6 \\
8 & 4 & 2 \\
5 & 3 & 1\\
\end{ytableau}
$\rightsquigarrow$ 
reverse alphabet
$\rightsquigarrow$ 
$E(T)=$ \begin{ytableau}
 1 & 2 & 6 \\
 3 & 4 & 7 \\
 5 & 9 & 11 \\
 8 & 10 & 12 \\
\end{ytableau}

\vspace{.5cm}
    \caption{Example of Lemma~\ref{PP}, showing the original rectangular standard Young tableau $T$, the 180 degree rotation  $\rho(T)$, and the evacuation $E(T)$. }
    \label{fig:evac-flip}
\end{figure}

\section{Multicolored NCMs and Reflection} \label{SEC:MC NCM and Reflection}

In this section, we give an algorithm to construct noncrossing matchings on enumerated boundary points from rectangular standard Young tableaux. In the next section, we use these NCMS to build web graphs that correspond exactly to the rectangular standard Young tableaux.  The intermediate steps we describe in this section will be applied to define reflection on noncrossing matchings, and later on webs. 

\begin{definition}
    An \emph{arc} in $\{1, 2, \ldots, N\}$ is an ordered pair $(i,j)$ with $i<j$.  If $(i,j)$ is an arc, we refer to $i$ as the \emph{start} (or \emph{initial point}) of the arc, $j$ as the \emph{end} (or \emph{terminal point}) of the arc, and $i, j$ collectively as the  \emph{boundary points} of the arc. 
\end{definition}
    
    We often draw arcs as triangles or semicircles above the $x$-axis.

\begin{definition}  
    A \emph{matching} on $\{1, 2, \ldots, N\}$ is a set of arcs $\{(i_1, j_1), (i_2, j_2), \ldots, (i_k,j_k)\}$ so that the union of endpoints $\{i_1, i_2, \ldots, i_k\} \cup \{j_1, j_2, \ldots, j_k\} \subseteq \{1, 2, \ldots, N\}$ has cardinality $2k$. We say in this case that the arcs have \emph{disjoint boundary points}.
    
    Two arcs $(i,j)$ and $(i',j')$ are \emph{crossing} if and only if $i<i'<j<j'$ as shown in Figure~\ref{Fig:crossingVSnoncrossing}.  A \emph{noncrossing matching (NCM)} is a matching in which no two arcs cross. 
\end{definition}
    \begin{figure}[h]
 \begin{center}
     \begin{tikzpicture}[scale=0.7]
\draw[thick] (.5,0) -- (6.5,0);
\draw[thick] (8.5,0) -- (14.5,0);

\draw[thick] (2,0) to[out=90, in=180] (2.5,.5);
\draw[thick] (2.5,.5) to[out=0, in=90] (3,0);

\draw[thick] (1,0) to[out=90, in=180] (2.5,1);
\draw[thick] (2.5,1) to[out=0, in=90] (4,0);

\draw[thick] (5,0) to[out=90, in=180] (5.5,.5);
\draw[thick] (5.5,.5) to[out=0, in=90] (6,0);

  \node at (1, -0.5) {1};
    \node at (2, -0.5) {2};
    \node at (3, -0.5) {3};
    \node at (4,-0.5) {4};
    \node at (5,-0.5) {5};
    \node at (6, -0.5) {6};

    \node at (9, -0.5) {1};
    \node at (10, -0.5) {2};
    \node at (11, -0.5) {3};
    \node at (12,-0.5) {4};
    \node at (13,-0.5) {5};
    \node at (14, -0.5) {6};

\draw[thick] (9,0) to[out=90, in=180] (10,.75);
\draw[thick] (10,.75) to[out=0, in=90] (11,0);

\draw[thick] (10,0) to[out=90, in=180] (11,.75);
\draw[thick] (11,.75) to[out=0, in=90] (12,0);

\draw[thick] (13,0) to[out=90, in=180] (13.5,.5);
\draw[thick] (13.5,.5) to[out=0, in=90] (14,0);

\end{tikzpicture} \hspace{0.5in}
\end{center}

    \caption{Example illustrating an NCM (left) and a crossing matching (right), on the set $\{1,2,3,4,5,6\}$. The matching on the left has arcs $\{(1,4),(2,3), (5,6)\}$ and the matching on the right has arcs $\{(1,3),(2,4),(5,6)\}$.}
\label{Fig:crossingVSnoncrossing}  
\end{figure}
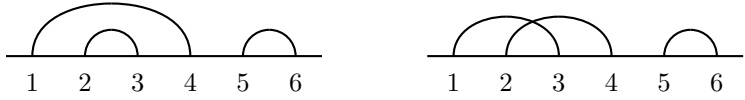

\begin{definition}   
    A \emph{multicolored matching} is an ordered set of matchings $\mathcal{M} = (\mathcal{M}_1, \mathcal{M}_2, \ldots, \mathcal{M}_{n-1})$ on $\{1, 2, \ldots, N\}$ so that each of $\{1, 2, \ldots, N\}$ appears,  for at least one $x$, as a boundary point $i, j$ of an arc $(i,j) \in \mathcal{M}_x$.  
    We say the arcs in $\mathcal{M}_x$ have \emph{color $x$}  for each $x=1, 2, \ldots, n-1$.
     A \emph{multicolored noncrossing matching (NCM)} is a collection of matchings $\mathcal{M} = (\mathcal{M}_1, \ldots, \mathcal{M}_{n-1})$ in which each $\mathcal{M}_x$ is  noncrossing.  
\end{definition}

Note that if $x \neq x'$ then the arcs in the matchings $\mathcal{M}_x$ and $\mathcal{M}_{x'}$ are permitted to cross each other.  We draw multicolored NCMs by associating a distinct color with each $x$.

\begin{figure}[h]
 \begin{center}
     \begin{tikzpicture}[scale=0.7]
\draw[thick] (.5,0) -- (8.5,0);

\draw[blue,thick] (2,0) to[out=90, in=180] (2.5,.5);
\draw[blue,thick] (2.5,.5) to[out=0, in=90] (3,0);
\draw[red,thick] (4,0) to[out=90, in=180] (4.5,.5);
\draw[red,thick] (4.5,.5) to[out=0, in=90] (5,0);
\draw[blue,thick] (1,0) to[out=90, in=180] (2.5,1);
\draw[blue,thick] (2.5,1) to[out=0, in=90] (4,0);
\draw[red,thick] (3,0) to[out=90, in=180] (5,1.5);
\draw[red,thick] (5,1.5) to[out=0, in=90] (7,0);
\draw[green,thick] (5,0) to[out=90, in=180] (5.5,.5);
\draw[green,thick] (5.5,.5) to[out=0, in=90] (6,0);
\draw[green,thick] (7,0) to[out=90, in=180] (7.5,.5);
\draw[green,thick] (7.5,.5) to[out=0, in=90] (8,0);

  \node at (1, -0.5) {1};
    \node at (2, -0.5) {2};
    \node at (3, -0.5) {3};
    \node at (4,-0.5) {4};
    \node at (5,-0.5) {5};
    \node at (6, -0.5) {6};
    \node at (7, -0.5) {7};
    \node at (8, -0.5) {8};

\end{tikzpicture} \hspace{0.5in}
\end{center}
    \caption{Example of a multicolored NCM $\mathcal{M}=\{\mathcal{M}_1,\mathcal{M}_2, \mathcal{M}_3\}$ on the set $\{1,2,3,4,5,6,7,8\}$, with matchings $\mathcal{M}_1=\{(1,4),(2,3)\}$, $\mathcal{M}_2=\{(3,7),(4,5)\}$, $\mathcal{M}_3 = \{(5,6), (7,8)\}$. We draw $\mathcal{M}_1$ in blue, $\mathcal{M}_2$ in red, and $\mathcal{M}_3$ in green. The arcs $(1,4)$ and $(3,7)$ can cross since $(1,4)\in \mathcal{M}_1$ and $(3,7)\in \mathcal{M}_2$.}
    \label{fig:crossing}
\end{figure}
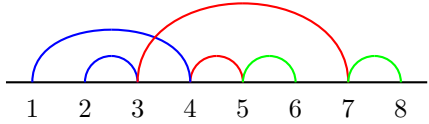

\subsection{Constructing a multicolored NCM $\mathcal{M}$ from a tableau $T$}

The next result gives a classical algorithm to create a multicolored NCM from a standard Young tableau.  We omit the proof because it is identical to that in~\cite{Tymoczko2012}.

\begin{lemma} \label{lemma: multicolored NCM from tableau}
Let $T$ be a standard Young tableau with $n$ rows.  For each row $x=1,2,\ldots,n-1,$ 
 \begin{enumerate}
            \item Let $j$ be the next available number reading left to right along row $x+1$.
            \item Find the rightmost number in row $x$ that is smaller than $j$  and has not yet been paired with a number on row $x+1$.  Say this number is $i$.
            \item Pair $i$ and $j$ and create arc $(i,j) \in \mathcal{M}_x$.
        \end{enumerate}   
Each $\mathcal{M}_x$ is a noncrossing matching, and the collection $\mathcal{M}^T = \{\mathcal{M}_1, \ldots, \mathcal{M}_{n-1}\}$ is  the multicolored noncrossing matching associated to the tableau $T$. 
\end{lemma}

Note that if $T$ is an $n \times k$ rectangle then each matching $\mathcal{M}_x$ for $x \in \{1, 2, \ldots, n-1\}$ has $k$ arcs.

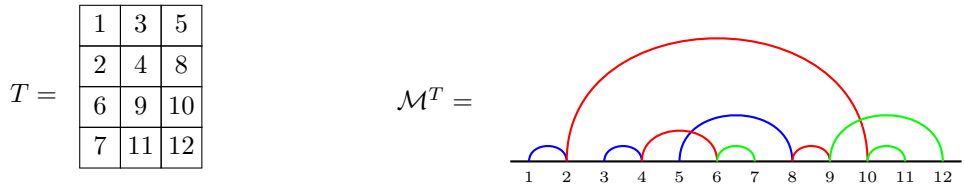
\begin{figure} [h]
    \centering
\begin{tabular}{@{}c@{\hskip 0.5em}c@{\hskip 0.5em}c@{\hskip 0.5em}c@{\hskip 0.5em}c@{\hskip 0.5em}c@{\hskip 0.5em}c@{}}

\begin{tikzpicture}[baseline]
  \node at (0,0) {$T=$};
\end{tikzpicture}

&

\scalebox{1}{%
\begin{ytableau}
1 & 3 & 5\\
2 & 4& 8\\
6 & 9 & 10\\
7 & 11 & 12\\
\end{ytableau}
}

&
\hspace{2cm}

&

\begin{tikzpicture}[thick, baseline=(current bounding box.center), xscale=0.5, yscale=.7]

\draw[thick] (.5,0) -- (12.5,0);
\foreach \i in {1,...,12} \node[below, font=\tiny] (\i) at (\i,0) {\i};
        \draw[blue, thick] (1) to[out=90,in=90] (2);
        \draw[blue, thick] (3) to[out=90,in=90] (4);
        \draw[blue, thick] (5) to[out=90,in=90] (8);
        \draw[red, thick] (2) to[out=90,in=90] (10);
        \draw[red, thick] (4) to[out=90,in=90] (6);
        \draw[red, thick] (8) to[out=90,in=90] (9);
        \draw[green, thick] (6) to[out=90,in=90] (7);
        \draw[green, thick] (9) to[out=90,in=90] (12);
        \draw[green, thick] (10) to[out=90,in=90] (11);

\node at (-1.5, 1.2) {$\mathcal{M}^T=$};
\end{tikzpicture}
\end{tabular}
    \caption{Example of Lemma~\ref{lemma: multicolored NCM from tableau}. The tableau $T$ on the left corresponds to the multicolored noncrossing matching $\mathcal{M}^T$ on the right, with $\mathcal{M}^T_1$ in blue, $\mathcal{M}^T_2$ in red, and $\mathcal{M}^T_3$ in green.  Label the four rows of the tableau from top to bottom. Reading left to right along row $2+1=3$, the first entry is $6$, so $6$ is matched with the rightmost entry in row $2$ that is less than $6$. This is entry $4$. Thus $\mathcal{M}^{T}_2$ (shown in red) contains the arc $(4,6)$.}
    \label{Fig:rectangluar T to NCM}
\end{figure}

When $T$ is a rectangular tableau,  the multicolored NCM $\mathcal{M}^T$ obtained from this process has a special property: each boundary point $i\in \{1, \ldots, N\}$ appears exactly once (if $i$ is the start of an arc in $\mathcal{M}_1$ or the end of an arc in $\mathcal{M}_{n-1}$) or twice (if $i$ is any other boundary point). More precisely, we have the following.

\begin{lemma} \label{lemma: rectangular tableau to multicolored NCM}
Let $T$ be a rectangular standard Young tableau with $N$ boxes and $n$ rows, and let $\mathcal{M}^T$ be the multicolored NCM obtained from $T$ by the algorithm in Lemma~\ref{lemma: multicolored NCM from tableau}.  Each arc $(i,j) \in \mathcal{M}_x$ satisfies:
\begin{enumerate}
    \item If $x>1$, the arc shares endpoint $i$ with exactly one other arc, and that arc is in $\mathcal{M}_{x-1}$.
    \item If $x<n-1$ the arc shares endpoint $j$ with exactly one other arc, and that arc is in $\mathcal{M}_{x+1}$.
    \item If $x=1$ then $i$ is on no other arc, and if $x=n-1$ then $j$ is on no other arc.
    \end{enumerate} 
\end{lemma}

\begin{proof}
By construction, each $\mathcal{M}_x$ has the same number of arcs as the number of boxes in the smaller of rows $x$ and $x+1$. Since $T$ is rectangular, every row has the same number of boxes, and this number is $k$. So every $\mathcal{M}_x$ has exactly $k$ arcs.

Also by construction, each number in row $x+1$ is used exactly once when constructing $\mathcal{M}_x$ and at most once when constructing $\mathcal{M}_{x+1}$. Since every row has the same number of boxes, this means each number in row $x+1$ is used exactly once when constructing $\mathcal{M}_{x+1}$. The only exceptions are at the extrema $x+1=1$ and $x+1=n$ since there is no $0^{th}$ row or $n^{th}$ matching.  The claim follows.
\end{proof}

The conditions in Lemma~\ref{lemma: rectangular tableau to multicolored NCM} are useful enough to warrant their own name.

\begin{definition} \label{def: standard rectangular matching}
A multicolored NCM $\mathcal{M}$ that satisfies (1)--(3) of Lemma~\ref{lemma: rectangular tableau to multicolored NCM} is called \emph{standard rectangular}.
\end{definition}

The argument in Lemma~\ref{lemma: rectangular tableau to multicolored NCM} is in fact reversible, meaning the map from standard Young tableaux to multicolored NCMs is bijective with the set of standard rectangular multicolored NCMs.

\subsection{Reflecting a multicolored NCM}
We now describe how to reflect a multicolored NCM.  The process is intuitive: we reflect over the vertical line $x=\frac{N+1}{2}$ and recolor the arcs. 

\begin{definition} \label{definition: reflecting multicolored NCM}
Suppose $(i,j)$ is an arc with color $x$.  Its \emph{reflection} is the arc $(N+1 - j, N+1-i)$ with color $n - x$.  The \emph{reflection} of the multicolored NCM $\mathcal{M}$ is denoted $\varphi(\mathcal{M})$ and defined as the multicolored NCM obtained by reflecting each arc in $\mathcal{M}$.
\end{definition}

Observe that the arc $(N+1-j, N+1-i)$ is the reflection of $(i,j)$ across the line $x=\frac{N+1}{2}$.

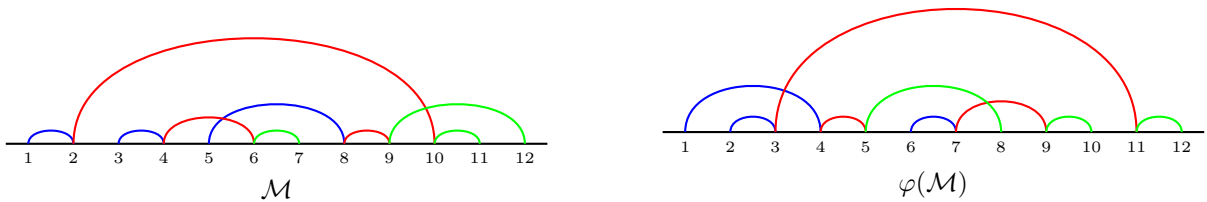
\begin{figure}[h]
     \begin{tikzpicture}[xscale=0.6, yscale=.6]

        \draw[thick] (.5,0) -- (12.5,0);
        \foreach \i in {1,...,12} \node[below, font=\tiny] (\i) at (\i,0) {\i};
        \draw[blue, thick] (1) to[out=90,in=90] (2);
        \draw[blue, thick] (3) to[out=90,in=90] (4);
        \draw[blue, thick] (5) to[out=90,in=90] (8);
        \draw[red, thick] (2) to[out=90,in=90] (10);
        \draw[red, thick] (4) to[out=90,in=90] (6);
        \draw[red, thick] (8) to[out=90,in=90] (9);
        \draw[green, thick] (6) to[out=90,in=90] (7);
        \draw[green, thick] (9) to[out=90,in=90] (12);
        \draw[green, thick] (10) to[out=90,in=90] (11);

    \node at (6.5,-1) {$\mathcal{M}$};

\end{tikzpicture} \hspace{0.5in}
    \begin{tikzpicture}[xscale=0.6,yscale=.7]
\draw[thick] (.5,0) -- (12.5,0);
        \foreach \i in {1,...,12} \node[below, font=\tiny] (\i) at (\i,0) {\i};
        \draw[blue, thick] (1) to[out=90,in=90] (4);
        \draw[blue, thick] (2) to[out=90,in=90] (3);
        \draw[blue, thick] (6) to[out=90,in=90] (7);
        \draw[red, thick] (3) to[out=90,in=90] (11);
        \draw[red, thick] (4) to[out=90,in=90] (5);
        \draw[red, thick] (7) to[out=90,in=90] (9);
        \draw[green, thick] (5) to[out=90,in=90] (8);
        \draw[green, thick] (9) to[out=90,in=90] (10);
        \draw[green, thick] (11) to[out=90,in=90] (12);

\node at (6.5,-1) {$\varphi(\mathcal{M})$};
\end{tikzpicture}

\caption{Example of Definition~\ref{definition: reflecting multicolored NCM}, in which $n=4$, $k=3$ and $N=12$, using blue for $\mathcal{M}_1$, red for $\mathcal{M}_2$, and green for $\mathcal{M}_3$. The reflection of arc $(1,2)$ is given by the map: $$(1,2) \mapsto (12+1-2, 12+1-1) = (11,12).$$ Arc $(1,2)\in \mathcal{M}$ is blue (color $1$) so its reflection $(11,12)\in \varphi(\mathcal{M)}$ is color $4-1=3$, or green.} 
\label{FIG: mulitcolored NCM and reflected Multicolored NCM}
\end{figure}

\section{$\mathfrak{sl}_n$ webs and Reflection} \label{SEC:webs and reflection}

We first recall Fontaine's definition of $\mathfrak{sl}_n$ web graphs (see also \cite{RT2025}). We then extend the algorithm described in \cite{Tymoczko2012} to build webs from standard Young tableaux,  and finally define the operation of reflection on the resulting web graphs.

\begin{definition}
    An $\mathfrak{sl}_n$-web is a planar, directed, edge-weighted graph with a boundary such that:
    \begin{itemize}
        \item every boundary vertex is univalent,
        \item every interior vertex is trivalent,
        \item every edge $e$ has a wedge weight $\ell(e) \in \{1,2,\ldots, n-1\}$; and
        \item and at each interior vertex $v$, the weights of the edges incident to $v$ \emph{preserve flow mod $n$}, meaning
        \[\sum_{\textup{ edges into }v} \ell(e) - \sum_{\textup{ edges out of }v} \ell(e) \equiv 0 \mod n.\]
    \end{itemize}
\end{definition}

In the literature, the integer assigned to a given edge is often referred to as an \textit{edge label}. To avoid ambiguity, we instead use the term \textit{\bf edge weight} for the integer assigned to a directed edge in a web graph and use the word \textit{\bf color} to refer to the integer assigned to a given arc in a multicolored NCM.

\subsection{Constructing a web $w$ from a multicolored NCM $\mathcal{M}$}

We describe a function that produces a web graph $w$ from a multicolored NCM $\mathcal{M}$.  The first two steps create the underlying graph, extending the algorithm in, e.g., \cite{Tymoczko2012} or \cite{RT2025}. The final steps direct and weight the edges. The proof is analogous to that of \cite{Tymoczko2012}, so we only sketch it.

\begin{lemma} \label{lemma: web from NCM algorithm}
Fix a positive number $m$. Suppose $\mathcal{M}$ is a standard rectangular multicolored NCM, as in Definition~\ref{def: standard rectangular matching}.  Draw each arc in $\mathcal{M}$ as a triangle formed by line segments of slopes $m$ and $-m$ above the $x$-axis in the plane.  The following algorithm produces a web graph $\widetilde{w_\mathcal{M}}$.
\begin{enumerate}
    \item {\bf At each boundary vertex $j$ where two arcs $(i, j)$ and $(j, k)$ are incident:} add a $Y$ shape, comprised of a single vertical edge at the vertex $j$ whose weight is $1$, and a new interior vertex from which the two arcs diverge. (See Figure~\ref{inserty}.) 
    
    \begin{figure}[h]
   \begin{tikzpicture}
    \draw[thick] (1,0) -- (3,0);
    \draw[blue,thick,] (1.5,1) --(2,0);
    \draw[blue,thick,-<] (2,0) -- (1.75,.5);
    \draw[red,thick,] (2,0) -- (2.5,1); 
    \draw[red,thick, ->] (2,0) -- (2.25,.5) ;
    
    \draw[thick][->] (3.7,.5) -- (4.3,.5);
    \draw[thick] (5,0) -- (7,0);
    \draw[thick][->] (6,0) -- (6,.35);
    \draw[thick][-] (6,.33) -- (6,.5);
    \draw[thick] (5.5,1) -- (6,.5);
    \draw[thick,->] (5.5,1)  -- (5.75,.75);
    \draw[thick] (6,.5) -- (6.5,1);
     \draw[thick,->] (6,.5) -- (6.325,.825);
    \node at (5.7,.27) {1};
   \end{tikzpicture}
       \caption{Y replacement with blue arc $(i,j)\in\mathcal{M}_x$ and red arc $(j,k)\in\mathcal{M}_{x+1}$.}     \label{inserty}

   \end{figure}
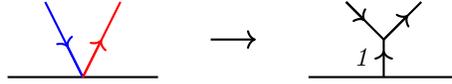

    \item {\bf At each interior vertex $v$ where two arcs $(i,j)$ and $(i',j')$ cross:} replace with a dumbbell, or \textit{I} shape, comprised of a new interior vertex $v'$ and a vertical edge $vv'$ between the two vertices. To decide which arc segments are on top and bottom, suppose $i < i' < j < j'$, or equivalently a clockwise trajectory along arc $(i,j)$ moves \emph{into} $(i',j')$ while a clockwise trajectory along $(i',j')$ moves \emph{out of} $(i,j)$. Then vertex $v$ is incident to edges on the arc segments to $i$ and $j'$ while $v'$ is incident to edges on the arc segments to $i'$ and $j$.  (See Figure~\ref{dumbell}.)
    
       \begin{figure}[h]
   \begin{tikzpicture}[scale=.8]
   
    \draw[red,thick][->] (0,0) -- (1.5,1.5);
    \draw[blue,thick][->] (0,1.5) -- (1.5,0);
    \node[blue, scale=.9] at (-.25,1.75) {$i$};
    \node[blue, scale=.9] at (1.75,0) {$j$};
    \node[red, scale=.9] at (-.25,0) {$i'$};
    \node[red, scale=.9] at (1.85,1.75) {$j'$};
    
    \draw[thick, scale=1.4][->] (2,.5) -- (2.6,.5);
    
    \begin{scope}[xshift=2cm,]
    
     \draw[thick][->] (4,.5) -- (4,.8);
     
    \draw[thick] (4,.75) -- (4,1);
    
    \draw[thick] (4,1) -- (4.5,1.5);
    \draw[thick][->] (4,1) -- (4.325,1.325);
    \draw[thick] (3.5,1.5) -- (4,1);
    \draw[thick][->] (3.5,1.5) --(3.75,1.25);
    
    \draw[thick][->] (3.5,0) -- (3.75,.25);
    
    \draw[thick] (3.5,0) -- (4,.5);
    
    \draw[thick] (4,.5) -- (4.5,0);
    \draw[thick][->] (4,.5) -- (4.325,.175);
    
    \node[scale=.9] at (3.2,0) {$x'$};
    \node[scale=.9] at (4.85,.8) {$x'-x$};
    \node[scale=.9] at (3.2,1.6) {$x$};
    \node[scale=.9] at (4.8,1.6) {$x'$};
    \node[scale=.9] at (4.8,0) {$x$};
    \end{scope}
    \end{tikzpicture}
    \caption{Dumbbell crossing of blue arc $(i, j) \in \mathcal{M}_x$ and red arc $(i', j') \in \mathcal{M}_{x'}$, with $i<i'<j<j'$ and $x' > x$.}    
    \label{dumbell}
\end{figure}
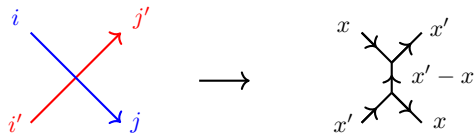
    \item {\bf Directing and weighting edges created from a single arc:} each edge $e$ created from a single arc should be directed consistent with the clockwise direction of that arc and weighted $\ell(e)=x$ if the arc is in matching $\mathcal{M}_x$. 
    \item {\bf Directing and weighting new edges created from two arcs:} Each new edge created in a Y or dumbbell can be thought of as having exactly two arcs running over it, in opposite directions relative to the edge if following a clockwise trajectory along each arc. Suppose one arc is in $\mathcal{M}_x$ and the other is in $\mathcal{M}_{x'}$, and without loss of generality assume $x<x'$.  Then direct the edge consistent with the clockwise trajectory along the arc in $\mathcal{M}_{x'}$ and weight the edge with $x'-x$. This means in Figure~\ref{inserty} the inserted vertical edge is always weighted $(x+1)-x=1$, and in Figure~\ref{dumbell} the inserted vertical edge weight $\ell(vv')$ is given by $ \ell(vv')= x'-x $.
\end{enumerate}
\end{lemma}

\begin{proof}
    Note that if each arc $(i,j)$ is drawn as a triangle with vertices $(i,0)$, $(j,0)$, and $(\frac{j-i}{2}, \frac{m(j-i)}{2})$, then two arcs intersect in at most one point.  Moreover, if two arcs intersect in a point, then this intersection is on the edge of slope $m$ for one of the arcs and $-m$ for the other. Consequently, no point lies on the intersection of more than two arcs.  Finally, no arcs in $\mathcal{M}_x$ intersect each other.  Each step of the algorithm produces a local change on the graph, restricted to a small neighborhood, so the graph $\widetilde{w_{\mathcal{M}}}$ is planar.  
    
    We now confirm that every vertex of $\widetilde{w_{\mathcal{M}}}$ has the appropriate degree. Every point on the intersection of two arcs in $\mathcal{M}$ is either a boundary point or an interior point. An interior intersection point in $\mathcal{M}$ becomes two trivalent vertices in $\widetilde{w_{\mathcal{M}}}$, while a boundary intersection point becomes one trivalent vertex incident to a boundary edge in $\widetilde{w_{\mathcal{M}}}$.  In both cases, the resulting graph $\widetilde{w_{\mathcal{M}}}$ only has trivalent interior vertices and univalent boundary vertices.

    Finally, each interior vertex $v$ was constructed incident to one edge with one arc entering and another edge with a different arc exiting (both along their clockwise trajectories), plus a third edge with both of those arcs in opposite directions.  Suppose without loss of generality that the arc entering $v$ is in $\mathcal{M}_x$ and the arc exiting $v$ is in $\mathcal{M}_{x'}$.  Then the corresponding edges in $\widetilde{w_{\mathcal{M}}}$ are weighted $x$ and oriented into $v$, respectively weighted $x'$ and oriented out.  Thus, their contribution to the flow modulo $n$ at $v$ is $x-x'$.  If $x < x'$ then the new edge is oriented into $v$ and weighted $x'-x$, while if $x > x'$ then the new edge is oriented out of $v$ and weighted $x-x'$.  In both cases, the flow modulo $n$ is zero, as desired.  So $\widetilde{w_{\mathcal{M}}}$ is in fact a web graph.
\end{proof}

One could extend Lemma~\ref{lemma: web from NCM algorithm} to more general $\mathcal{M}$, allowing boundary edges of arbitrary weight (with fewer conditions on arcs containing a boundary point $i$), and allowing multiple arcs to cross simultaneously. To prove that the output is still a web graph requires the machinery of stranding \cite{RT2025}, and in particular for certain arcs to run counterclockwise subject to careful constraints.

    It is customary to use the convention that boundary vertices in a web graph are sources.  Webs are considered \emph{equivalent} under the equivalence relation defined---for any particular edge---by simultaneously flipping the direction and changing the edge weight from $x$ to $n-x$.  (See \cite{RT2025} for a complete set of web graph equivalence relations.) Using this equivalence for the web graphs $\widetilde{w_{\mathcal{M}}}$ gives the following.

    \begin{corollary} \label{corollary: webs from tableau in standard form}
   Suppose $\mathcal{M}$ is a standard rectangular multicolored NCM and let $\widetilde{w_{\mathcal{M}}}$ be its associated web graph.  Denote by $w_{\mathcal{M}}$ the graph obtained from $\widetilde{w_{\mathcal{M}}}$ by, for each boundary vertex $u$,  replacing the edge $uv$ by an edge directed from $u$ to $v$ weighted $1$.
   
   Then $w_\mathcal{M}$ is a web graph equivalent to $\widetilde{w_{\mathcal{M}}}$ in the sense that it is obtained from $\widetilde{w_{\mathcal{M}}}$ by reversing some edge directions and replacing the weight $\ell(e)$ on each flipped edge $e$ by $n-\ell(e)$.  If $\mathcal{M}^T$ is associated to a standard Young tableau $T$ then we denote the web graph $w_{\mathcal{M}^T}$ by $w_T$.
    \end{corollary}

    \begin{proof}
        By Lemma~\ref{lemma: rectangular tableau to multicolored NCM}, each boundary point $1, 2, \ldots, N$ appears exactly once in $\mathcal{M}^T$ (if it starts an arc in $\mathcal{M}_1$ or ends an arc in $\mathcal{M}_{n-1}$) or twice (if any other boundary point), and the end point of each arc in $\mathcal{M}_x$ is the start of exactly one arc in $\mathcal{M}_{x+1}$ for $x=1,2,\ldots,n-2$. Hence, we have three kinds of boundary edges: edges which start an arc in $\mathcal{M}_1$, boundary edges in a Y replacement, and edges which end an arc in $\mathcal{M}_{n-1}$. We address each of these three cases separately.

        First, suppose the boundary edge is the start of an arc in $\mathcal{M}_1$.  In this case, the edge is directed up because of the clockwise convention, and weighted $1$ because the arc is in $\mathcal{M}_1$.  So these boundary edges are the same in $w_T$ as in $w_{\mathcal{M}^T}$. 
        
        Second, suppose the boundary vertex lies on one arc in $\mathcal{M}_x$ and one arc in $\mathcal{M}_{x+1}$.  Since arcs are directed clockwise, the arc from $\mathcal{M}_x$ is directed down into the boundary and the arc from $\mathcal{M}_{x+1}$ is directed up out of the boundary. By the algorithm in Lemma~\ref{lemma: web from NCM algorithm}, the boundary edge of the Y is directed with the arc whose color is a larger number, so it is directed out of the boundary with the arc in $\mathcal{M}_{x+1}$, and its weight is the difference $(x+1)-x = 1$ as desired. So these boundary edges, too, are the same in $w_T$ as in $w_{\mathcal{M}^T}$. 
         
        Third and finally, suppose the boundary edge is the end of an arc in $\mathcal{M}_{n-1}$. Again by convention, the end of each arc in $\mathcal{M}_{n-1}$ is directed down and weighted $n-1$. Flip these edges and replace the weight $n-1$ by $n-(n-1)=1$. These are the only edges that differ between $w_T$ and $w_{\mathcal{M}^T}$ and they are related according to the desired equivalence relation.  This proves the claim.
    \end{proof}

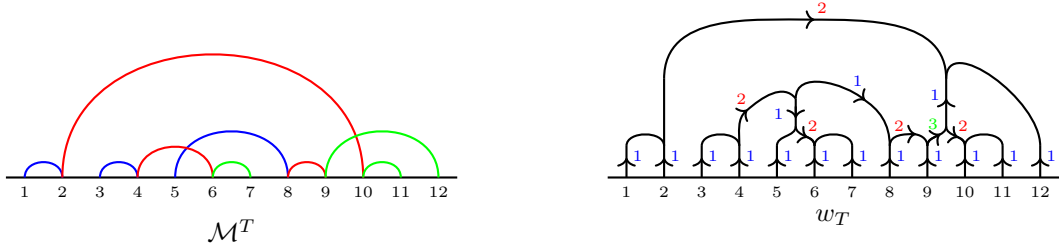
\begin{figure}[h]
     \begin{tikzpicture}[xscale=.5, yscale=.7]
\draw[thick] (.5,0) -- (12.5,0);
        \foreach \i in {1,...,12} \node[below, font=\tiny] (\i) at (\i,0) {\i};
        \draw[blue, thick] (1) to[out=90,in=90] (2);
        \draw[blue, thick] (3) to[out=90,in=90] (4);
        \draw[blue, thick] (5) to[out=90,in=90] (8);
        \draw[red, thick] (2) to[out=90,in=90] (10);
        \draw[red, thick] (4) to[out=90,in=90] (6);
        \draw[red, thick] (8) to[out=90,in=90] (9);
        \draw[green, thick] (6) to[out=90,in=90] (7);
        \draw[green, thick] (9) to[out=90,in=90] (12);
        \draw[green, thick] (10) to[out=90,in=90] (11);

   \node at (6.5,-1) {$\mathcal{M}^T$};

\begin{scope}[yscale=.75, xshift=16cm]
        \draw[thick] (.5,0) -- (12.5,0);
         \foreach \i in {1,...,12}
            \node[below, font=\tiny] (\i) at (\i,0) {\i};
        \foreach \i in {1,...,12}
            \draw[thick] (\i,.4) -- (\i,.8);
        \foreach \i in {1,...,12}
        \draw[thick,->] (\i,0) -- (\i,.5);

        \draw[thick] (2,.8) -- (2,2.5);

        \draw[thick] (1,.8) to[out=90,in=90] (2,.8);
        \draw[thick, ->] (2,2.5) to[out=90,in=180] (6,4);
        \draw[thick] (5.9,4) to[out=0,in=90] (9.5,2.5);
        
        \draw[thick, ->] (9.5,1.3) -- (9.5,2);
        \draw[thick] (9.5,1.9) -- (9.5,2.5);

        \draw[thick, -<] (10,.8) to[out=90,in=340] (9.6,1.095);
        \draw[thick] (10,.8) to[out=90,in=-90] (9.5,1.3);
        
        \draw[ thick] (3,.8) to[out=90,in=90] (4,.8);
        \draw[ thick] (6,.8) to[out=90,in=90] (7,.8);

        \draw[thick, <-] (5.5,1.5) -- (5.5,1.8);
        \draw[thick] (5.5,1.3) -- (5.5,2);

        \draw[ thick] (4.2, 1.7) to[out=40,in=130] (5.5,2);
        
        \draw[ thick, ->] (4, .8) to[out=90,in=220] (4.22,1.72);
        
        \draw[thick] (5.5,1.3) to[out=-90,in=90] (6,.8);
        \draw[thick , -<] (6,.8) to[out=90,in=340] (5.6,1.095);

        \draw[thick] (5,.8) to[out=90,in=-90] (5.5,1.3);
        \draw[thick] (7.25,2.01) to[out=-30,in=90] (8,.8);
        \draw[thick, ->] (5.5,2) to[out=90,in=142] (7.35,1.95);

         \draw[thick] (9.5,2.5) to[out=90,in=90] (12,.8);
         \draw[ thick] (10,.8) to[out=90,in=90] (11,.8);

         \draw[-> , thick] (8,.8) to[out=90,in=180] (8.7,1.1);
         \draw[ thick] (8.6,1.1) to[out=0,in=90] (9,.8);

         \draw[thick, ->] (9,.8) to[out=90,in=210]  (9.35,1.105);
        \draw[thick] (9,.8) to[out=90,in=-90] (9.5,1.3);

         \node[blue, font=\tiny] at (1.3,.55) {$1$};
         \node[blue, font=\tiny] at (3.3,.55) {$1$};
         \node[blue, font=\tiny] at (5.3,.55) {$1$};
         \node[blue, font=\tiny] at (2.3,.55) {$1$};
         \node[blue, font=\tiny] at (4.3,.55) {$1$};
         \node[blue, font=\tiny] at (6.3,.55) {$1$};
         \node[blue, font=\tiny] at (7.3,.55) {$1$};
         \node[blue, font=\tiny] at (9.3,.55) {$1$};
         \node[blue, font=\tiny] at (10.3,.55) {$1$};
         \node[blue, font=\tiny] at (11.3,.55) {$1$};
         \node[blue, font=\tiny] at (12.3,.55) {$1$};
         \node[blue, font=\tiny] at (8.3,.55) {$1$};
        \node[blue, font=\tiny] at (7.15,2.45) {$1$};

         \node[red, font=\tiny] at (6.2,4.3) {$2$};
         \node[red, font=\tiny] at (4.05,2.0) {$2$};
         \node[red, font=\tiny] at (8.25,1.3) {$2$};
         \node[red, font=\tiny] at (5.95,1.3) {$2$};

         \node[red, font=\tiny] at (9.95,1.3) {$2$};
         \node[green, font=\tiny] at (9.15,1.35) {$3$};

         \node[blue, font=\tiny] at (5.05,1.6) {$1$};
         \node[blue, font=\tiny] at (9.2,2.1) {$1$};

         \node at (6.5,-1) {$w_T$};

\end{scope}
\end{tikzpicture}

\caption{Example of Corollary~\ref{corollary: webs from tableau in standard form}. The multicolored NCM $\mathcal{M}^T$ on the left was obtained from $T$ in Figure~\ref{Fig:rectangluar T to NCM}. On the right, the web $w_T$ is obtained from $T$ via $\mathcal{M}^T$ according to Lemma~\ref{lemma: web from NCM algorithm} and Corollary~\ref {corollary: webs from tableau in standard form}.} 
\label{Fig: multicored NCM to web}
\end{figure}

\subsection{Reflecting a web}

We now define what it means to reflect a web, which is simply reflecting over the line $x = \frac{N+1}{2}$. Lemma~\ref{lemma: reflected web is a web}  confirms that the reflection of a web graph is itself a web graph. 

\begin{definition} \label{definition: reflection of a web}
Suppose that $w$ is a (Fontaine) web for $\mathfrak{sl}_n$.  Define the \emph{reflection} of $w$ to be the graph $\varphi(w)$ obtained by reflecting over the line $x = \frac{N+1}{2}$.  
\end{definition}

\begin{figure}[h]
\begin{tikzpicture}[scale=0.6]
\begin{scope}[xshift=16cm]
    \draw[dashed, red] (6.5,-.5) -- (6.5,4.5);
        \draw[thick] (.5,0) -- (12.5,0);
         \foreach \i in {1,...,12}
            \node[below, font=\tiny] (\i) at (\i,0) {\i};
        \foreach \i in {1,...,12}
            \draw[thick] (\i,.4) -- (\i,.8);
        \foreach \i in {1,...,12}
        \draw[thick,->] (\i,0) -- (\i,.5);

        \draw[thick] (2,.8) -- (2,2.5);

        \draw[thick] (1,.8) to[out=90,in=90] (2,.8);
        \draw[thick, ->] (2,2.5) to[out=90,in=180] (6,4);
        \draw[thick] (5.9,4) to[out=0,in=90] (9.5,2.5);
        
        \draw[thick, ->] (9.5,1.3) -- (9.5,2);
        \draw[thick] (9.5,1.9) -- (9.5,2.5);

        \draw[thick, -<] (10,.8) to[out=90,in=340] (9.6,1.095);
        \draw[thick] (10,.8) to[out=90,in=-90] (9.5,1.3);
        
        \draw[ thick] (3,.8) to[out=90,in=90] (4,.8);
        \draw[ thick] (6,.8) to[out=90,in=90] (7,.8);

        \draw[thick, <-] (5.5,1.5) -- (5.5,1.8);
        \draw[thick] (5.5,1.3) -- (5.5,2);

        \draw[ thick] (4.2, 1.7) to[out=40,in=130] (5.5,2);
        
        \draw[ thick, ->] (4, .8) to[out=90,in=220] (4.22,1.72);
        
        \draw[thick] (5.5,1.3) to[out=-90,in=90] (6,.8);
        \draw[thick , -<] (6,.8) to[out=90,in=340] (5.6,1.095);

        \draw[thick] (5,.8) to[out=90,in=-90] (5.5,1.3);
        \draw[thick] (7.25,2.01) to[out=-30,in=90] (8,.8);
        \draw[thick, ->] (5.5,2) to[out=90,in=142] (7.35,1.95);

         \draw[thick] (9.5,2.5) to[out=90,in=90] (12,.8);
         \draw[ thick] (10,.8) to[out=90,in=90] (11,.8);

         \draw[-> , thick] (8,.8) to[out=90,in=180] (8.7,1.1);
         \draw[ thick] (8.6,1.1) to[out=0,in=90] (9,.8);

         \draw[thick, ->] (9,.8) to[out=90,in=210]  (9.35,1.105);
        \draw[thick] (9,.8) to[out=90,in=-90] (9.5,1.3);

         \node[blue, font=\tiny] at (1.3,.55) {$1$};
         \node[blue, font=\tiny] at (3.3,.55) {$1$};
         \node[blue, font=\tiny] at (5.3,.55) {$1$};
         \node[blue, font=\tiny] at (2.3,.55) {$1$};
         \node[blue, font=\tiny] at (4.3,.55) {$1$};
         \node[blue, font=\tiny] at (6.3,.55) {$1$};
         \node[blue, font=\tiny] at (7.3,.55) {$1$};
         \node[blue, font=\tiny] at (9.3,.55) {$1$};
         \node[blue, font=\tiny] at (10.3,.55) {$1$};
         \node[blue, font=\tiny] at (11.3,.55) {$1$};
         \node[blue, font=\tiny] at (12.3,.55) {$1$};
         \node[blue, font=\tiny] at (8.3,.55) {$1$};
        \node[blue, font=\tiny] at (7.15,2.45) {$1$};

         \node[red, font=\tiny] at (6.2,4.3) {$2$};
         \node[red, font=\tiny] at (4.05,2.0) {$2$};
          \node[red, font=\tiny] at (8.25,1.3) {$2$};
        \node[red, font=\tiny] at (5.95,1.3) {$2$};

         \node[red, font=\tiny] at (9.9,1.35) {$2$};
         \node[green, font=\tiny] at (9.2,1.35) {$3$};

         \node[blue, font=\tiny] at (5.05,1.6) {$1$};
         \node[blue, font=\tiny] at (9.2,2.1) {$1$};

         \node at (6.5,-1) {$w_T$};

\end{scope}
\end{tikzpicture} \hspace{0.25in}
\raisebox{-.2em}{\begin{tikzpicture}[xscale=-1, scale=0.6]
\begin{scope}[xshift=16cm]
    \draw[dashed, red] (6.5,-.5) -- (6.5,4.5);
        \draw[thick] (.5,0) -- (12.5,0);
         \foreach \i in {1,...,12}
            \node[below, font=\tiny] (\i) at (13-\i,0) {\i};
        \foreach \i in {1,...,12}
            \draw[thick] (\i,.4) -- (\i,.8);
        \foreach \i in {1,...,12}
        \draw[thick,->] (\i,0) -- (\i,.5);

        \draw[thick] (2,.8) -- (2,2.5);

        \draw[thick] (1,.8) to[out=90,in=90] (2,.8);
        \draw[thick, ->] (2,2.5) to[out=90,in=180] (6,4);
        \draw[thick] (5.9,4) to[out=0,in=90] (9.5,2.5);
        
        \draw[thick, ->] (9.5,1.3) -- (9.5,2);
        \draw[thick] (9.5,1.9) -- (9.5,2.5);

        \draw[thick, -<] (10,.8) to[out=90,in=340] (9.6,1.095);
        \draw[thick] (10,.8) to[out=90,in=-90] (9.5,1.3);
        
        \draw[ thick] (3,.8) to[out=90,in=90] (4,.8);
        \draw[ thick] (6,.8) to[out=90,in=90] (7,.8);

        \draw[thick, <-] (5.5,1.5) -- (5.5,1.8);
        \draw[thick] (5.5,1.3) -- (5.5,2);

        \draw[ thick] (4.2, 1.7) to[out=40,in=130] (5.5,2);
        
        \draw[ thick, ->] (4, .8) to[out=90,in=220] (4.22,1.72);
        
        \draw[thick] (5.5,1.3) to[out=-90,in=90] (6,.8);
        \draw[thick , -<] (6,.8) to[out=90,in=340] (5.6,1.095);

        \draw[thick] (5,.8) to[out=90,in=-90] (5.5,1.3);
        \draw[thick] (7.25,2.01) to[out=-30,in=90] (8,.8);
        \draw[thick, ->] (5.5,2) to[out=90,in=142] (7.35,1.95);

         \draw[thick] (9.5,2.5) to[out=90,in=90] (12,.8);
         \draw[ thick] (10,.8) to[out=90,in=90] (11,.8);

         \draw[-> , thick] (8,.8) to[out=90,in=180] (8.7,1.1);
         \draw[ thick] (8.6,1.1) to[out=0,in=90] (9,.8);

         \draw[thick, ->] (9,.8) to[out=90,in=210]  (9.35,1.105);
        \draw[thick] (9,.8) to[out=90,in=-90] (9.5,1.3);

         \node[blue, font=\tiny] at (1.3,.55) {$1$};
         \node[blue, font=\tiny] at (3.3,.55) {$1$};
         \node[blue, font=\tiny] at (5.3,.55) {$1$};
         \node[blue, font=\tiny] at (2.3,.55) {$1$};
         \node[blue, font=\tiny] at (4.3,.55) {$1$};
         \node[blue, font=\tiny] at (6.3,.55) {$1$};
         \node[blue, font=\tiny] at (7.3,.55) {$1$};
         \node[blue, font=\tiny] at (9.3,.55) {$1$};
         \node[blue, font=\tiny] at (10.3,.55) {$1$};
         \node[blue, font=\tiny] at (11.3,.55) {$1$};
         \node[blue, font=\tiny] at (12.3,.55) {$1$};
         \node[blue, font=\tiny] at (8.3,.55) {$1$};
        \node[blue, font=\tiny] at (7.15,2.45) {$1$};

         \node[red, font=\tiny] at (6.2,4.3) {$2$};
         \node[red, font=\tiny] at (4.05,2.0) {$2$};
          \node[red, font=\tiny] at (8.25,1.3) {$2$};
        \node[red, font=\tiny] at (5.95,1.3) {$2$};

         \node[red, font=\tiny] at (9.9,1.35) {$2$};
         \node[green, font=\tiny] at (9.2,1.35) {$3$};

         \node[blue, font=\tiny] at (5.05,1.6) {$1$};
         \node[blue, font=\tiny] at (9.2,2.1) {$1$};

         \node at (6.5,-1) {$\varphi(w_T)$};

\end{scope}
\end{tikzpicture}}
    \caption{ Example of Definition~\ref{definition: reflection of a web}, showing the reflection of the previous web graph.}
\end{figure}
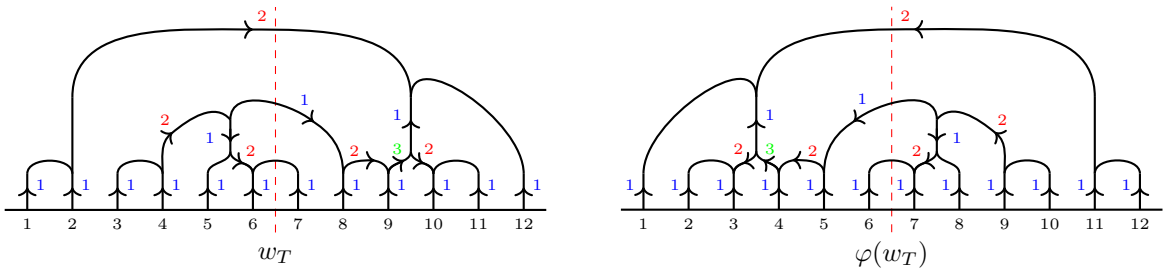

We note that reflection of a web can be viewed equivalently as reflection over the $y$-axis, then shifting the boundary vertices right according to the map $i \rightarrow N+1+1$. This horizontal translation $i \rightarrow N+1+i$ corresponds directly to the last step of evacuation in Definition~\ref{definition: evacuation}, when the integer $N+1$ is added to each entry in the tableau.

\begin{lemma} \label{lemma: reflected web is a web}
    If $w$ is a web graph for $\mathfrak{sl}_n$ then the reflected graph $\varphi(w)$ is also a web graph for $\mathfrak{sl}_n$.  Moreover, if $w_T$ is the web graph obtained from a rectangular standard Young tableau as in Corollary~\ref{corollary: webs from tableau in standard form}, then the boundary edges of $\varphi(w_T)$ are all directed out of the boundary and weighted $1$.
\end{lemma}

\begin{proof}
Reflecting over a vertical line preserves the property of being a planar directed graph, and also preserves the degree of each vertex as well as whether edges incident to that vertex are directed in or out.  The relative positions of each edge about an interior vertex change after reflection, but the edge weights remain the same so the flow into and out of each interior vertex remains unchanged. Thus, the flow at each interior vertex is $0$ modulo $n$ as required.  Further, if $i$ is a boundary vertex in $w$ then its image in $\varphi(w)$ is $N+1-i$ so the boundary vertices of $\varphi(w)$ are again $1, 2, \ldots, N$. We conclude that reflecting a Fontaine web produces a graph that is also a Fontaine web. 

Finally, if $w_T$ is obtained from a rectangular standard Young tableau as in Corollary~\ref{corollary: webs from tableau in standard form}, then the boundary edges of $w_T$ are all directed out of the boundary and weighted $1$. In general, the sequence of edge weights reading left to right along the boundary in the reflection $\varphi(w)$ is the same as the sequence of edge weights reading right to left along the boundary in $w$. In the case of $w_T$, the sequence of wedge weights along the boundary is symmetric, in particular, all $1$. Thus, the sequence of edge weights along the boundary in $\varphi(w_T)$ is also all $1$ (and all oriented out of the boundary). This completes the proof.
\end{proof}

\section{Evacuation and reflection} \label{SEC: Evacuation and reflection}

Our main claim is that the maps from Figure~\ref{commutative} in the introduction commute.  In other words, the reflection of the multicolored NCM associated to a tableau $T$ is the multicolored NCM associated to the evacuation $E(T)$, and the reflection of the web associated to $T$ is---up to equivalence under specific edge-flips---the web associated to the evacuation $E(T)$.  In this section, we show that each square of the diagram in Figure~\ref{commutative} commutes.  We start with the left square and then proceed to the right square.

\subsection{The left square of the diagram in Figure~\ref{commutative} commutes}

We prove this in two steps, via an intermediate object.  First, we associate a multicolored NCM to the rotation of a tableau, and prove that it has the same arcs as $\mathcal{M}_T$ but with colors flipped $x \leftrightarrow n-x$.  Then, we show that reflecting this intermediate object (while preserving colors) gives the multicolored NCM $\mathcal{M}_{E(T)}$ associated to the evacuation. This is the body of the proof of the left square of our commutative diagram.

We now give a protocol to pair integers $1, 2, \ldots, n$ in a rotated tableau that is essentially the same as the usual rule to create a multicolored NCM, but reading the tableau ``in a mirror and upside-down."

\begin{definition}\label{definition: rotated matching}
    Given a rectangular standard Young tableau $T$ with rotated tableau $\rho(T)$, we denote the rotated multicolored NCM by $\mathcal{M}^{\rho(T)}$ and define it by the following protocol.
    \begin{itemize}
        \item For each row $x = 1, 2, \ldots, n-1$ in $\rho(T)$:
        \begin{enumerate}
            \item Let $j$ be the next number reading \emph{right to left} along row $x$.
            \item Find the \emph{leftmost} number in row $x+1$ that is smaller than $j$ and has not yet been paired with a number on row $x$.  Say this number is $i$.
            \item Pair $i$ and $j$ and create arc $(i,j) \in \mathcal{M}^{\rho(T)}_x$.
        \end{enumerate}   
    \end{itemize}
\end{definition}

We give an example of Definition~\ref{definition: rotated matching} in Figure~\ref{Fig:rotatedT to NCM}.

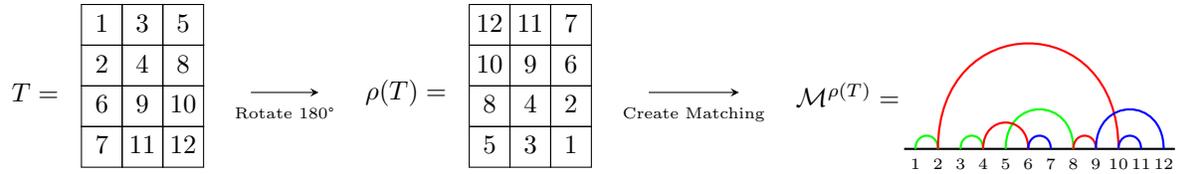
\begin{figure}[h]
    \centering
\begin{tabular}{@{}c@{\hskip 0.5em}c@{\hskip 0.5em}c@{\hskip 0.5em}c@{\hskip 0.5em}c@{\hskip 0.5em}c@{\hskip 0.5em}c@{}}

\begin{tikzpicture}[baseline]
  \node at (0,0) {$T=$};
\end{tikzpicture}

&

\scalebox{1}{%
\begin{ytableau}
1 & 3 & 5\\
2 & 4& 8\\
6 & 9 & 10\\
7 & 11 & 12\\
\end{ytableau}
}

&

\begin{tikzpicture}[baseline]
  \draw[-stealth] (0,0) -- (0.9,0);
  \node[below=2pt] at (0.45,0) {\tiny Rotate 180\textdegree};
\end{tikzpicture}

&

\begin{tikzpicture}[baseline]
  \node at (0,0) {$\rho(T)=$};
\end{tikzpicture}

&

\scalebox{1}{%
\begin{ytableau}
 12 & 11 & 7 \\
10 & 9 & 6 \\
8 & 4 & 2 \\
5 & 3 & 1\\
\end{ytableau}
}

&

\begin{tikzpicture}[baseline]
  \draw[-stealth] (0,0) -- (1.2,0);
  \node[below=2pt] at (0.6,0) {\tiny Create Matching};
\end{tikzpicture}

&

\begin{tikzpicture}[thick, baseline=(current bounding box.center), xscale=0.3, yscale=0.6]


\node at (-2,1.2) {$\mathcal{M}^{\rho(T)}=$};

 \draw[thick] (.5,0) -- (12.5,0);
        \foreach \i in {1,...,12} \node[below, font=\tiny] (\i) at (\i,0) {\i};
        \draw[green, thick] (1) to[out=90,in=90] (2);
        \draw[green, thick] (3) to[out=90,in=90] (4);
        \draw[green, thick] (5) to[out=90,in=90] (8);
        \draw[red, thick] (2) to[out=90,in=90] (10);
        \draw[red, thick] (4) to[out=90,in=90] (6);
        \draw[red, thick] (8) to[out=90,in=90] (9);
        \draw[blue, thick] (6) to[out=90,in=90] (7);
        \draw[blue, thick] (9) to[out=90,in=90] (12);
        \draw[blue, thick] (10) to[out=90,in=90] (11);

\end{tikzpicture}
\end{tabular}
    \caption{Example of Definition~\ref{definition: rotated matching}. On the left is a tableau $T$, in the middle its rotation $\rho(T)$, and on the right the multicolored NCM $\mathcal{M}^{\rho(T)}$. Label the rows of $\rho(T)$ top to bottom $1,2,3, 4$. Reading right to left along row $1$, the first entry is $7$. The left-most entry in row $1+1=2$ that is less than $7$ is $6$. Thus arc $(6,7)$ is in $\mathcal{M}^{\rho(T)}_1$ (shown in blue).}
    \label{Fig:rotatedT to NCM}
\end{figure}

Comparing $\mathcal{M}^{\rho(T)}$ in Figure~\ref{Fig:rotatedT to NCM} to $\mathcal{M}^T$ in Figure~\ref{Fig: multicored NCM to web}, we observe that the two multicolored NCMs have the same arcs, but arcs colored blue in $\mathcal{M}^T$ are green in $\mathcal{M}^{\rho(T)}$ and vice versa.  This is true in general, as the next lemma demonstrates.

\begin{lemma} \label{lemma: comparing rotated NCM to original NCM}
We claim the following:
\begin{enumerate}
    \item For each row $x$ we have $\mathcal{M}^{\rho(T)}_x  = \mathcal{M}^T_{n-x}$
    \item Each $\mathcal{M}^{\rho(T)}_x $ is a noncrossing matching. 
    \item The set $\mathcal{M}^{\rho(T)} = (\mathcal{M}^{\rho(T)}_1, \ldots, \mathcal{M}^{\rho(T)}_{n-1})$ is a multicolored NCM.
\end{enumerate}
\end{lemma}

\begin{proof}
    Rotating $T$ by 180 degrees exchanges right and left as well as top and bottom.  This means that reading left-to-right in $T$ is the same as reading right-to-left in $\rho(T)$, that row $(n+1)-x$ in $T$ has the same numbers (in reverse order) as row $x$ in $\rho(T)$, and that comparing row $(n+1)-x$ to the row immediately above it in $T$ is the same as comparing row $x$ to the row immediately below it in $\rho(T)$. Thus our protocol for creating $\mathcal{M}^{\rho(T)}_x$ corresponds precisely to the (rotated) rules for creating $\mathcal{M}^T_{n-x}$.  We conclude that for each $x=1, 2, \ldots, n-1$ we have $\mathcal{M}^{\rho(T)}_x = \mathcal{M}^T_{n-x}$.  Parts (2) and (3) follow immediately so the claim holds.
\end{proof}

The next lemma is very similar but compares the rotated tableau $\rho(T)$ to the evacuated tableau $E(T)$. 

\begin{lemma} \label{lemma: comparing rotated matching to evacuated matching}
Let $\mathcal{M}^{E(T)}$ denote the multicolored NCM corresponding to the evacuation $E(T)$. We have
\[(i,j) \in \mathcal{M}^{\rho(T)}_x \hspace{0.5in} \textup{ if and only if } \hspace{0.5in} (N-j+1,N-i+1) \in \mathcal{M}^{E(T)}_x\]
\end{lemma}

\begin{proof}
We use Patrias and Pechenik's result from Lemma~\ref{lemma: Patrias Pechenik}: the evacuation $E(T)$ is obtained from the rotation $\rho(T)$ by replacing each entry $i$ in $\rho(T)$ according to the map $i \rightarrow N-i+1$.

The matching $\mathcal{M}^{E(T)}_x$ looks left-to-right along row $x+1$, pairing the rightmost number in row $x$ that is smaller and unpaired.  This is equivalent to reading right-to-left along row $x$, pairing the leftmost number in row $x+1$ that is greater and unpaired. By definition,  the matching $\mathcal{M}^{\rho(T)}_x$  defined on the rotated tableau looks right-to-left along row $x$, pairing the leftmost number in row $x+1$ that is smaller and unpaired. So $(i,j) \in \mathcal{M}^{\rho(T)}_x$ if and only if $(N-j+1, N-i+1) \in \mathcal{M}^{E(T)}_x$ because $i<j$ exactly when $N-j+1 < N-i+1$.  This completes the proof.
\end{proof}

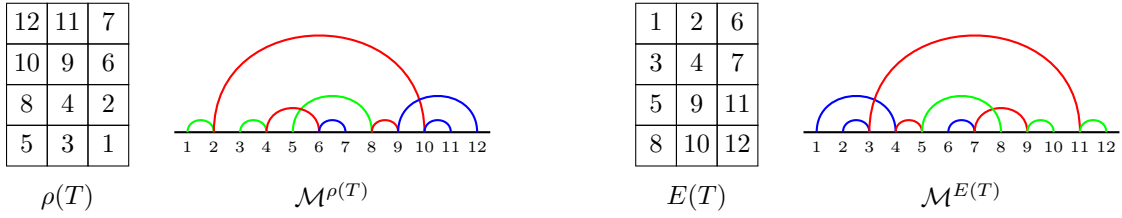
\begin{figure}[h]
\begin{center}
\begin{tabular}{cccc}
\raisebox{0.6in}{\begin{tikzpicture}[baseline=(current bounding box.north)]
    \begin{scope}[scale=1.2]
        \node (phi) {
         \begin{ytableau}
 12 & 11 & 7 \\
10 & 9 & 6 \\
8 & 4 & 2 \\
5 & 3 & 1\\
\end{ytableau}
        };
    \end{scope}
    \node at (0,-1.5) {$\rho(T)$};
\end{tikzpicture}
}

& 
\begin{tikzpicture}[baseline=(current bounding box.center)]
    \begin{scope}[xscale=0.35, yscale=0.55]
        \node at (6.5,-1.5) {$\mathcal{M}^{\rho(T)}$};
        \draw[thick] (.5,0) -- (12.5,0);
        \foreach \i in {1,...,12} \node[below, font=\tiny] (\i) at (\i,0) {\i};
        \draw[green, thick] (1) to[out=90,in=90] (2);
        \draw[green, thick] (3) to[out=90,in=90] (4);
        \draw[green, thick] (5) to[out=90,in=90] (8);
        \draw[red, thick] (2) to[out=90,in=90] (10);
        \draw[red, thick] (4) to[out=90,in=90] (6);
        \draw[red, thick] (8) to[out=90,in=90] (9);
        \draw[blue, thick] (6) to[out=90,in=90] (7);
        \draw[blue, thick] (9) to[out=90,in=90] (12);
        \draw[blue, thick] (10) to[out=90,in=90] (11);
    \end{scope}
\end{tikzpicture} \hspace{0.5in}
&
\raisebox{0.6in}{\begin{tikzpicture}[baseline=(current bounding box.north)]
    \begin{scope}[scale=1.2]
        \node (E) {
\begin{ytableau}
 1 & 2 & 6 \\
3 & 4 & 7 \\
5 & 9 & 11 \\
8 & 10 & 12\\
\end{ytableau}
        };
    \end{scope}
    \node at (0,-1.5) {$E(T)$};
\end{tikzpicture} }
&
\begin{tikzpicture}[baseline=(current bounding box.center)]
    \begin{scope}[xscale=0.35, yscale=0.55]
        \node at (6.5,-1.5) {$\mathcal{M}^{E(T)}$};
        \draw[thick] (.5,0) -- (12.5,0);
        \foreach \i in {1,...,12} \node[below, font=\tiny] (\i) at (\i,0) {\i};
        \draw[blue, thick] (1) to[out=90,in=90] (4);
        \draw[blue, thick] (2) to[out=90,in=90] (3);
        \draw[blue, thick] (6) to[out=90,in=90] (7);
        \draw[red, thick] (3) to[out=90,in=90] (11);
        \draw[red, thick] (4) to[out=90,in=90] (5);
        \draw[red, thick] (7) to[out=90,in=90] (9);
        \draw[green, thick] (5) to[out=90,in=90] (8);
        \draw[green, thick] (9) to[out=90,in=90] (10);
        \draw[green, thick] (11) to[out=90,in=90] (12);
    \end{scope}
\end{tikzpicture}
\\[1.2em]
\end{tabular}
\end{center}
\caption{Example of Lemma~\ref{lemma: comparing rotated matching to evacuated matching} using the tableaux from Figure~\ref{fig:evac-flip}.} \label{fig: compare rotated to evacuated tab and matching}
\end{figure}

The next corollary combines Lemma~\ref{lemma: comparing rotated NCM to original NCM} and Lemma~\ref{lemma: comparing rotated matching to evacuated matching}.  It also proves one half of our main result.

\begin{corollary}
For each arc, we have
\[ (i,j) \in   \mathcal{M}^T_x \hspace{0.5in} \textup{ if and only if } \hspace{0.5in} (N-j+1,N-i+1) \in \mathcal{M}^{E(T)}_{n-x}\]
In other words, the left square of the diagram in Figure~\ref{commutative} commutes.
\end{corollary}

\subsection{The right square of the diagram in Figure~\ref{commutative} commutes}

In principal, we need to show that creating a web graph from $\mathcal{M}$ and then reflecting the web gives the same output as reflecting $\mathcal{M}$ and then creating its web graph. In practice, after we create a web, we can imagine retaining the multicolored NCM superimposed on the graph and then reflecting both simultaneously. This is the key idea in our proof that reflection commutes with the process of building a web graph from a multicolored NCM. It is also an application of the notion of stranding introduced in \cite{RT2025}.

\begin{lemma} \label{lemma: right square commutes}
    Suppose $\mathcal{M}$ is a standard rectangular multicolored NCM as in Definition~\ref{def: standard rectangular matching}. If $w_{\mathcal{M}}$ is its web graph, then the reflected web graph satisfies 
    \[\varphi(w_{\mathcal{M}}) = w_{\varphi(\mathcal{M})} \hspace{0.3in} \textup{as undirected, unweighted planar graphs}. \] 
    Moreover, let $\sim$ denote the equivalence of web graphs under the edge-flips described in Corollary~\ref{corollary: webs from tableau in standard form}.  Then
    \[\varphi(w_{\mathcal{M}}) \sim w_{\varphi(\mathcal{M})}\]    
    In particular, suppose $T$ is a standard Young tableau with evacuation $E(T)$ and corresponding web graphs $w_T$ and $w_{E(T)}$. Let $\mathcal{E}_T$ denote the collection of internal edges in $w_T$ constructed in Step (3) of Lemma~\ref{lemma: web from NCM algorithm} and let $\varphi(w_T)_{\varphi(\mathcal{E}_T)}$ denote the web graph obtained from $\varphi(w_T)$ by performing edge-flips on the edges $\varphi(\mathcal{E}_T)$ as in Corollary~\ref{corollary: webs from tableau in standard form}.  Then the reflection satisfies $\varphi(\mathcal{E}_T) = \mathcal{E}_{E(T)}$ and \[\varphi(w_T)_{\varphi(\mathcal{E}_T)} = \varphi\big( (w_T)_{\mathcal{E}_T}\big) = w_{E(T)} \hspace{0.3in} \textup{ as web graphs.} \]
    In particular, the rightmost square of the diagram in Figure~\ref{commutative} commutes.
\end{lemma}
\begin{proof}
    For the following argument, when we create the web graph $w_{\mathcal{M}}$,  we will retain the colored arcs of $\mathcal{M}$ superimposed onto the edges of the web $w_\mathcal{M}$. Definitions~\ref{definition: reflecting multicolored NCM} and~\ref{definition: reflection of a web} are compatible, in the sense \emph{both} that the reflection $\varphi(\mathcal{M})$ will also consist of arcs superimposed on the reflected web graph $\varphi(w_\mathcal{M})$ \emph{and} that resolving $\varphi(\mathcal{M})$ into a web graph via the algorithm in Lemma~\ref{lemma: web from NCM algorithm} produces a graph $w_{\varphi(\mathcal{M})}$ whose underlying undirected and unweighted graph agrees with $\varphi(w_{\mathcal{M}})$.  This proves the first claim.
    
    We now compare directions and weights of edges in $\varphi(w_{\mathcal{M}})$ versus $w_{\varphi(\mathcal{M})}$. By Definition~\ref{definition: reflection of a web}, the weight of an edge $e$ in $w_{\mathcal{M}}$ is the same as that of its reflection $\varphi(e)$ in $\varphi(w_{\mathcal{M}})$ though reflection may change the direction of the edge.  Comparing with Definition~\ref{definition: reflecting multicolored NCM}, we see that each arc stays bound to the edge it was on, so each reflected edge has the same number of arcs superimposed.  However, for each edge $e$ in $w_{\mathcal{M}}$ and each arc $(i,j)$ of color $x$ that is superimposed on $e$, the reflected edge $\varphi(e)$ has arc $(N+1-j, N+1-i)$ colored $n-x$  superimposed in the reflected graph $\varphi(w_{\mathcal{M}})$.  In general, the color $n-x$ and clockwise direction of $(N+1-j, N+1-i)$ determines the weight and direction of $\varphi(e)$ in $w_{\varphi(\mathcal{M})}$. 
    
    There are three types of edges in the web graph $w_{\varphi(\mathcal{M})}$. We analyze them separately.
    \begin{enumerate}
        \item \emph{Boundary edges:} All boundary edges in $w_{\mathcal{M}}$ and in $w_{\varphi(\mathcal{M})}$ are weighted $1$ and directed out of the boundary. Reflection over a vertical line preserves weight and direction of these edges.  So the boundary edges of $\varphi(w_{\mathcal{M}})$ are all weighted one and directed out of the boundary, too.
        \item \emph{Non-boundary edges created from a single arc:} We saw above that each reflected edge $\varphi(e)$ in $\varphi(w_{\mathcal{M}})$ has the same number of arcs superimposed as $e$ did in $w_{\mathcal{M}}$.  Moreover, we saw that if $e$ in $w_{\mathcal{M}}$ has arc $(i,j)$ colored $x$ then $\varphi(e)$ has arc $(N+1-j, N+1-i)$ colored $n-x$.  The clockwise direction of $(N+1-j, N+1-i)$ is opposite that of $(i,j)$ so edge $\varphi(e)$ in  $w_{\varphi(\mathcal{M})}$ has opposite direction to $\varphi(e)$ in $\varphi(w_{\mathcal{M}})$ and also is weighted $n-x$ instead of $x$.  In other words, edge $\varphi(e)$ is related via the equivalence of edge-flip in $w_{\varphi(\mathcal{M})}$ versus $\varphi(w_{\mathcal{M}})$.
        \item \emph{Non-boundary edges created from a pair of arcs:} Any edge $e$ in $w_\mathcal{M}$ created from two different colored arcs, say colors $x<x'$, is sent to an edge $\varphi(e)$ in $\varphi(w_\mathcal{M})$ also with two colors superimposed but of colors $n-x > n-x'$. Again by Lemma~\ref{lemma: web from NCM algorithm}, the edge $e$ in $w_\mathcal{M}$ is directed with the clockwise direction of the arc colored $x'$ and weighted $x'-x$ while $\varphi(e)$ in $w_{\varphi(\mathcal{M})}$ is directed with the clockwise direction of the arc colored $n-x$ and weighted $(n-x)-(n-x')=x'-x$. After reflecting, the edge $\varphi(e)$ in $\varphi(w_\mathcal{M})$ follows the counterclockwise direction of the arc colored $n-x'$ but maintains the original weight $x'-x$.  The arcs colored $n-x'$ and $n-x$ run along edge $\varphi(e)$ in opposite directions, so the counterclockwise direction of the arc colored $n-x'$ is the clockwise direction of the arc $n-x$.  Thus the edge $\varphi(e)$ in $\varphi(w_\mathcal{M})$ agrees in both weight and direction  with the edge $\varphi(e)$ in $w_{\varphi(\mathcal{M})}$.  
    \end{enumerate}
    In other words, the web graph $w_{\varphi(\mathcal{M})}$ agrees with that of the reflected web graph $\varphi(w_{\mathcal{M}})$ up to the equivalence relation of edge-flips.  Specifically, let $\mathcal{E_M}$ be the collection of internal edges in $w_{\mathcal{M}}$ created from a single arc (as in Step (3) of Lemma~\ref{lemma: web from NCM algorithm}).  Then reflection sends edges $\mathcal{E_M}$ in $w_{\mathcal{M}}$ to $\mathcal{E}_{\varphi(\mathcal{M})}$ in both $w_{\varphi(\mathcal{M})}$ and $\varphi(w_{\mathcal{M}})$.  This proves the third claim. Specializing to matchings obtained from tableaux shows that the right square of the diagram in Figure~\ref{commutative} commutes, as desired.
\end{proof}

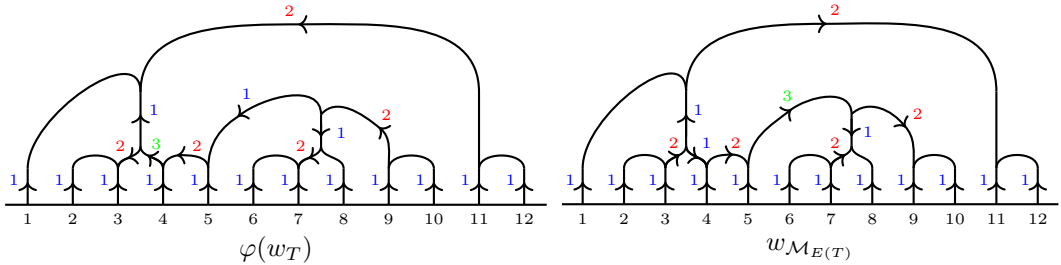
\begin{figure}[h]
\begin{center}
\raisebox{-.1em}{\begin{tikzpicture}[xscale=-1, scale=0.6]
\begin{scope}[xshift=16cm]
        \draw[thick] (.5,0) -- (12.5,0);
         \foreach \i in {1,...,12}
            \node[below, font=\tiny] (\i) at (13-\i,0) {\i};
        \foreach \i in {1,...,12}
            \draw[thick] (\i,.4) -- (\i,.8);
        \foreach \i in {1,...,12}
        \draw[thick,->] (\i,0) -- (\i,.5);

        \draw[thick] (2,.8) -- (2,2.5);

        \draw[thick] (1,.8) to[out=90,in=90] (2,.8);
        \draw[thick, ->] (2,2.5) to[out=90,in=180] (6,4);
        \draw[thick] (5.9,4) to[out=0,in=90] (9.5,2.5);
        
        \draw[thick, ->] (9.5,1.3) -- (9.5,2);
        \draw[thick] (9.5,1.9) -- (9.5,2.5);

        \draw[thick, -<] (10,.8) to[out=90,in=340] (9.6,1.095);
        \draw[thick] (10,.8) to[out=90,in=-90] (9.5,1.3);
        
        \draw[ thick] (3,.8) to[out=90,in=90] (4,.8);
        \draw[ thick] (6,.8) to[out=90,in=90] (7,.8);

        \draw[thick, <-] (5.5,1.5) -- (5.5,1.8);
        \draw[thick] (5.5,1.3) -- (5.5,2);

        \draw[ thick] (4.2, 1.7) to[out=40,in=130] (5.5,2);
        
        \draw[ thick, ->] (4, .8) to[out=90,in=220] (4.22,1.72);
        
        \draw[thick] (5.5,1.3) to[out=-90,in=90] (6,.8);
        \draw[thick , -<] (6,.8) to[out=90,in=340] (5.6,1.095);

        \draw[thick] (5,.8) to[out=90,in=-90] (5.5,1.3);
        \draw[thick] (7.25,2.01) to[out=-30,in=90] (8,.8);
        \draw[thick, ->] (5.5,2) to[out=90,in=142] (7.35,1.95);

         \draw[thick] (9.5,2.5) to[out=90,in=90] (12,.8);
         \draw[ thick] (10,.8) to[out=90,in=90] (11,.8);

         \draw[-> , thick] (8,.8) to[out=90,in=180] (8.7,1.1);
         \draw[ thick] (8.6,1.1) to[out=0,in=90] (9,.8);

         \draw[thick, ->] (9,.8) to[out=90,in=210]  (9.35,1.105);
        \draw[thick] (9,.8) to[out=90,in=-90] (9.5,1.3);

         \node[blue, font=\tiny] at (1.3,.55) {$1$};
         \node[blue, font=\tiny] at (3.3,.55) {$1$};
         \node[blue, font=\tiny] at (5.3,.55) {$1$};
         \node[blue, font=\tiny] at (2.3,.55) {$1$};
         \node[blue, font=\tiny] at (4.3,.55) {$1$};
         \node[blue, font=\tiny] at (6.3,.55) {$1$};
         \node[blue, font=\tiny] at (7.3,.55) {$1$};
         \node[blue, font=\tiny] at (9.3,.55) {$1$};
         \node[blue, font=\tiny] at (10.3,.55) {$1$};
         \node[blue, font=\tiny] at (11.3,.55) {$1$};
         \node[blue, font=\tiny] at (12.3,.55) {$1$};
         \node[blue, font=\tiny] at (8.3,.55) {$1$};
        \node[blue, font=\tiny] at (7.15,2.45) {$1$};

         \node[red, font=\tiny] at (6.2,4.3) {$2$};
         \node[red, font=\tiny] at (4.05,2.0) {$2$};
         \node[red, font=\tiny] at (8.25,1.3) {$2$};
         \node[red, font=\tiny] at (5.95,1.3) {$2$};

         \node[red, font=\tiny] at (9.95,1.3) {$2$};
         \node[green, font=\tiny] at (9.15,1.35) {$3$};

         \node[blue, font=\tiny] at (5.05,1.6) {$1$};
         \node[blue, font=\tiny] at (9.2,2.1) {$1$};

         \node at (6.5,-1) {$\varphi(w_T)$};

\end{scope}
\end{tikzpicture}}
\begin{tikzpicture}[baseline=(current bounding box.south)]
 
\begin{scope}[xscale=0.55, yscale=0.6]
        \draw[thick] (.5,0) -- (12.5,0);
        \foreach \i in {1,...,12}
            \node[below, font=\tiny] (\i) at (\i,0) {\i};
        \foreach \i in {1,...,12}
            \draw[thick] (\i,.4) -- (\i,.8);
        \foreach \i in {1,...,12}
        \draw[thick,->] (\i,0) -- (\i,.5);

        \draw[thick] (11,.8) -- (11,2.5);

        \draw[thick] (12,.8) to[out=90,in=90] (11,.8);
        
        \draw[thick] (6.8,4) to[out=0,in=90] (11,2.5);
        \draw[thick, ->] (3.5,2.5) to[out=90,in=180] (6.9,4);
        
        \draw[thick, ->] (3.5,1.3) -- (3.5,2);
        \draw[thick] (3.5,1.5) -- (3.5,2.5);

        \draw[thick, ->] (3,.8) to[out=90,in=210]  (3.35,1.105);
        \draw[thick] (3,.8) to[out=90,in=-90] (3.5,1.3);
        
        \draw[ thick] (10,.8) to[out=90,in=90] (9,.8);
        \draw[ thick] (7,.8) to[out=90,in=90] (6,.8);

        \draw[thick, <-] (7.5,1.5) -- (7.5,1.9);
        \draw[thick] (7.5,1.3) -- (7.5,2);
        
        \draw[thick, ->] (7.5,2) to[out=40,in=130]  (8.81, 1.7);
        
        \draw[ thick] (8.78,1.74) to[out=-45,in=90]  (9, .8);
        
        \draw[thick, ->] (7,.8) to[out=90,in=210]  (7.35,1.105);
        \draw[thick] (7,.8) to[out=90,in=-90] (7.5,1.3);
        
        \draw[thick] (8,.8) to[out=90,in=-90] (7.5,1.3);

        \draw[thick, ->] (5,.8) to[out=90,in=215] (6.05,2.105);
        \draw[thick] (6,2.075) to[out=35,in=90] (7.5,2);

        \draw[thick] (4,.8) to[out=90,in=-90] (3.5,1.3);
        \draw[thick] (3.5,2.5) to[out=90,in=90] (1,.8);
        \draw[thick] (3,.8) to[out=90,in=90] (2,.8);

        \draw[->, thick] (4,.8) to[out=90,in=180] (4.65,1.1);
        \draw[thick] (4.62,1.1) to[out=0,in=90] (5,.8);

        \draw[thick, -<] (4,.8) to[out=90,in=340] (3.6,1.095);
        \draw[thick] (4,.8) to[out=90,in=-90] (3.5,1.3);

         \node[blue, font=\tiny] at (0.7,.55) {$1$};
         \node[red, font=\tiny] at (3.2,1.35) {$2$};
         \node[blue, font=\tiny] at (3.8,2.1) {$1$};
         \node[blue, font=\tiny] at (2.7,.55) {$1$};
         \node[blue, font=\tiny] at (1.7,.55) {$1$};
         \node[blue, font=\tiny] at (3.7,.55) {$1$};
         \node[blue, font=\tiny] at (4.7,.55) {$1$};
        \node[red, font=\tiny] at (4.7,1.35) {$2$};
        \node[blue, font=\tiny] at (4,1.35) {$1$};
        \node[green, font=\tiny] at (5.95,2.4) {$3$};
        \node[blue, font=\tiny] at (7.7,.55) {$1$};
        \node[blue, font=\tiny] at (5.7,.55) {$1$};
        \node[blue, font=\tiny] at (6.7,.55) {$1$};
        \node[blue, font=\tiny] at (8.7,.55) {$1$};
     \node[blue, font=\tiny] at (9.7,.55) {$1$};
     \node[blue, font=\tiny] at (10.7,.55) {$1$};
     \node[blue, font=\tiny] at (11.7,.55) {$1$};
     \node[red, font=\tiny] at (7.1,1.35) {$2$};
      \node[blue, font=\tiny] at (7.9,1.6) {$1$};
      \node[red, font=\tiny] at (9.1,2.0) {$2$};
       \node[red, font=\tiny] at (7.1,4.3) {$2$};
       \node at (6.5,-1) {$w_{\mathcal{M}_{E(T)}}$};         
\end{scope}
\end{tikzpicture}
\end{center}
\caption{Example of Lemma~\ref{lemma: right square commutes} continuing with the previous web graph.} \label{fig: reflected to evacuated web graph}
\end{figure}

Figure~\ref{fig: reflected to evacuated web graph} completes our running example of a $\mathfrak{sl}_4$ web graph, comparing the reflected web graph $\varphi(w_T)$ to the web graph $w_{E(T)}$ for the evacuated tableau.  We end with Figure~\ref{fig: sl3 example}, which gives an example of this for $\mathfrak{sl}_3$, which was analyzed by Patrias and Pechenik  \cite{PatriasPechenik2023}.  This motivates our last corollary, showing that if we use the edge-weight and direction conventions that are customary for $\mathfrak{sl}_3$ and $\mathfrak{sl}_4$ web graphs, then the reflected web graph is precisely the web graph associated to the evacuation.

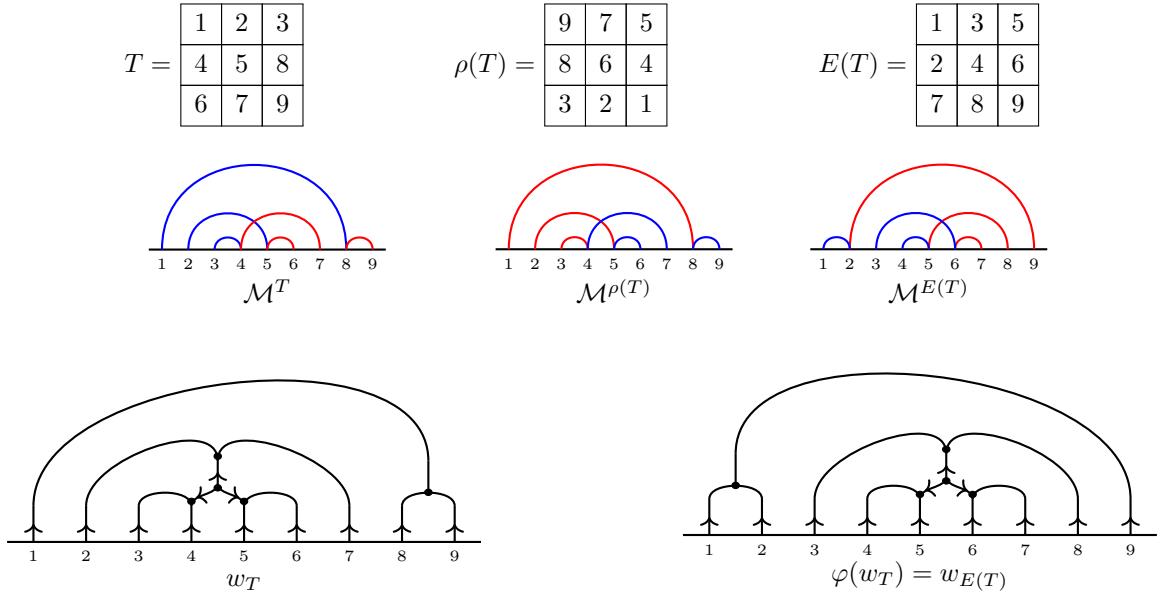
\begin{figure}[h!]
\begin{center}
\begin{tabular}{ccc}

\begin{tikzpicture}[baseline=(current bounding box.north)]
    \begin{scope}[scale=1.2]
        \node (T) {
            $T=$
            \begin{ytableau}
                1 & 2 & 3 \\
                4 & 5 & 8 \\
                6 & 7 & 9 \\
            \end{ytableau}
        };
    \end{scope}
\end{tikzpicture}
\hspace{1.3cm}
&
\begin{tikzpicture}[baseline=(current bounding box.north)]
    \begin{scope}[scale=1.2]
        \node (phi) {
            $\rho(T)=$
            \begin{ytableau}
                9 & 7 & 5 \\
                8 & 6 & 4 \\
                3 & 2 & 1 \\
            \end{ytableau}
        };
    \end{scope}
\end{tikzpicture}
\hspace{1.3cm}

&
\begin{tikzpicture}[baseline=(current bounding box.north)]
    \begin{scope}[scale=1.2]
        \node (E) {
            $E(T)=$
            \begin{ytableau}
                1 & 3 & 5 \\
                2 & 4 & 6 \\
                7 & 8 & 9 \\
            \end{ytableau}
        };
    \end{scope}
\end{tikzpicture}

\\

\begin{tikzpicture}[baseline=(current bounding box.center)]
    \begin{scope}[xscale=0.35, yscale=0.55]
        \node at (5,-1) {$\mathcal{M}^T$};
        \draw[thick] (.5,0) -- (9.5,0);
        \foreach \i in {1,...,9} \node[below, font=\tiny] (\i) at (\i,0) {\i};
        
        \draw[blue, thick] (1) to[out=90,in=90] (8);
        \draw[blue, thick] (2) to[out=90,in=90] (5);
        \draw[blue, thick] (3) to[out=90,in=90] (4);
        \draw[red, thick] (4) to[out=90,in=90] (7);
        \draw[red, thick] (5) to[out=90,in=90] (6);
        \draw[red, thick] (8) to[out=90,in=90] (9);
        
    \end{scope}
\end{tikzpicture}
&
\begin{tikzpicture}[baseline=(current bounding box.center)]
    \begin{scope}[xscale=0.35, yscale=0.55]
        \node at (5,-1) {$\mathcal{M}^{\rho(T)}$};
        \draw[thick] (.5,0) -- (9.5,0);
        \foreach \i in {1,...,9} \node[below, font=\tiny] (\i) at (\i,0) {\i};
        
        \draw[red, thick] (1) to[out=90,in=90] (8);
        \draw[red, thick] (2) to[out=90,in=90] (5);
        \draw[red, thick] (3) to[out=90,in=90] (4);
        \draw[blue, thick] (4) to[out=90,in=90] (7);
        \draw[blue, thick] (5) to[out=90,in=90] (6);
        \draw[blue, thick] (8) to[out=90,in=90] (9);
        
    \end{scope}
\end{tikzpicture}
&
\begin{tikzpicture}[baseline=(current bounding box.center)]
    \begin{scope}[xscale=-0.35, yscale=0.55]
        \node at (5,-1) {$\mathcal{M}^{E(T)}$};
        \draw[thick] (.5,0) -- (9.5,0);
        \foreach \i in {1,...,9} \node[below, font=\tiny] (\i) at (\i,0) {};
        \node[below, font=\tiny] at (9,0) {$1$};
        \node[below, font=\tiny] at (8,0) {$2$};
        \node[below, font=\tiny] at (7,0) {$3$};
        \node[below, font=\tiny] at (6,0) {$4$};
        \node[below, font=\tiny] at (5,0) {$5$};
        \node[below, font=\tiny] at (4,0) {$6$};
        \node[below, font=\tiny] at (3,0) {$7$};
        \node[below, font=\tiny] at (2,0) {$8$};
        \node[below, font=\tiny] at (1,0) {$9$};
        
        \draw[red, thick] (1) to[out=90,in=90] (8);
        \draw[red, thick] (2) to[out=90,in=90] (5);
        \draw[red, thick] (3) to[out=90,in=90] (4);
        \draw[blue, thick] (4) to[out=90,in=90] (7);
        \draw[blue, thick] (5) to[out=90,in=90] (6);
        \draw[blue, thick] (8) to[out=90,in=90] (9);
        
    \end{scope}
\end{tikzpicture}
\end{tabular}
\end{center}
\begin{center}
\begin{tikzpicture}[xscale=0.7, yscale=0.6]
     \draw[thick] (.5,0) -- (9.5,0);
         \foreach \i in {1,...,9}
            \node[below, font=\tiny] (\i) at (\i,0) {\i};

    \node at (5,-.9) {$w_T$};
    
    \filldraw[black] (4,.9) circle (2pt);
    \filldraw[black] (5,.9) circle (2pt);
    \filldraw[black] (4.5,1.2) circle (2pt);
    \filldraw[black] (4.5,1.9) circle (2pt);
    \filldraw[black] (8.5,1.1) circle (2pt);

    \draw[thick] (3,.8) to[out=90,in=90] (4,.8);
    \draw[thick] (5,.8) to[out=90,in=90] (6,.8);
    \draw[thick] (8,.8) to[out=90,in=90] (9,.8);

    \draw[thick] (2,.8) to[out=90,in=90] (4.5,1.8);
    \draw[thick] (7,.8) to[out=90,in=90] (4.5,1.8);

    \draw[thick] (8.5,1.05) -- (8.5,1.8);
    \draw[thick] (1,.8) to[out=90,in=90] (8.5,1.8);

    \begin{scope}[thick,decoration={
    markings,
    mark=at position 0.5 with {\arrow{>}}}
    ] 

    \foreach \i in {1,...,9}
            \draw[thick, postaction={decorate}] (\i,0) -- (\i,.8);
    \draw[thick, postaction={decorate}] (4.5, 1.3) -- (4.5, 1.8);
    \end{scope}

    \begin{scope}[thick,decoration={
    markings,
    mark=at position 0.5 with {\arrow{<}}}
    ] 
    \draw[thick, postaction={decorate}] (5,.8) to[out=90,in=-90] (4.5,1.3);
    \draw[thick, postaction={decorate}] (4,.8) to[out=90,in=-90] (4.5,1.3);

    \end{scope}
\end{tikzpicture}\hspace{1in}
\begin{tikzpicture}[xscale=-0.7, yscale=0.6]
     \draw[thick] (.5,0) -- (9.5,0);
         \foreach \i in {1,...,9}
            \node[below, font=\tiny] (\i) at (10-\i,0) {\i};

    \node at (5,-.9) {$\varphi(w_T) = w_{E(T)}$};

    \filldraw[black] (4,.9) circle (2pt);
    \filldraw[black] (5,.9) circle (2pt);
    \filldraw[black] (4.5,1.2) circle (2pt);
    \filldraw[black] (4.5,1.9) circle (2pt);
    \filldraw[black] (8.5,1.1) circle (2pt);

    \draw[thick] (3,.8) to[out=90,in=90] (4,.8);
    \draw[thick] (5,.8) to[out=90,in=90] (6,.8);
    \draw[thick] (8,.8) to[out=90,in=90] (9,.8);

    \draw[thick] (2,.8) to[out=90,in=90] (4.5,1.8);
    \draw[thick] (7,.8) to[out=90,in=90] (4.5,1.8);

    \draw[thick] (8.5,1.05) -- (8.5,1.8);
    \draw[thick] (1,.8) to[out=90,in=90] (8.5,1.8);

    \begin{scope}[thick,decoration={
    markings,
    mark=at position 0.5 with {\arrow{>}}}
    ] 

    \foreach \i in {1,...,9}
            \draw[thick, postaction={decorate}] (\i,0) -- (\i,.8);
    \draw[thick, postaction={decorate}] (4.5, 1.3) -- (4.5, 1.8);
    \end{scope}

    \begin{scope}[thick,decoration={
    markings,
    mark=at position 0.5 with {\arrow{<}}}
    ] 
    \draw[thick, postaction={decorate}] (5,.8) to[out=90,in=-90] (4.5,1.3);
    \draw[thick, postaction={decorate}] (4,.8) to[out=90,in=-90] (4.5,1.3);

    \end{scope}
\end{tikzpicture}
\end{center}
\caption{Example of main result in $\mathfrak{sl}_3$ using the usual $\mathfrak{sl}_3$ web graph conventions that all interior vertices are sources or sinks and all edges are weighted $1$.  Note that with the usual $\mathfrak{sl}_3$ edge direction and weights, the reflection $\varphi(w_T)$ is exactly the graph $w_{E(T)}$.} \label{fig: sl3 example}
\end{figure}

\begin{corollary} \label{corollary: sl3 and sl4 web conventions}
For $\mathfrak{sl}_3$ and $\mathfrak{sl}_4$ web graphs, it is customary to direct and weight edges so that:
\begin{itemize}
    \item every boundary vertex is a source and the incident edge is weighted $1$,
    \item every interior vertex is incident to three edges weighted $1$ (for $\mathfrak{sl}_3$) or two edges weighted $1$ and one edge weighted $2$ (for $\mathfrak{sl}_4$),  
    \item edges of weight $1$ are directed so every interior vertex is a source or a sink, and
    \item (for $\mathfrak{sl}_4$) edges of weight $2$ are undirected.
\end{itemize}
With these conventions and $T$ a standard Young tableau of shape $n \times k$ for $n=3$ or $n=4$, we have
\[\varphi(w_T) = w_{E(T)}\]
\end{corollary}

\begin{proof}
Reflection preserves the property of being a source or sink and reflection over a vertical line preserves the direction of boundary edges.  Lemma~\ref{lemma: right square commutes} showed that the undirected, unweighted graphs $\varphi(w_T)$ and $w_{E(T)}$ agree.  It follows that with the above conventions, the graphs $\varphi(w_T) = w_{E(T)}$ on the nose, without applying any edge-flips.
\end{proof}

\bibliographystyle{alpha}
\fontsize{10pt}{10pt}\selectfont
\newcommand{\etalchar}[1]{$^{#1}$}

\end{document}